\documentclass[11pt]{amsart}
\usepackage{xcolor,soul}

\usepackage{amsmath}
\usepackage{amssymb}

\usepackage{amsfonts}
\usepackage{amsthm}
\usepackage{enumerate}
\usepackage[mathscr]{eucal}

\usepackage{tikz}

\usepackage{bbm}
\usepackage{graphicx}
\usepackage{verbatim}
\usepackage{color}

\newcommand{\Be}{\begin{equation}}
\newcommand{\Ee}{\end{equation}}

\newcommand{\Bea}{\begin{eqnarray}}
\newcommand{\Eea}{\end{eqnarray}}
\newcommand{\Beas}{\begin{eqnarray*}}
\newcommand{\Eeas}{\end{eqnarray*}}
\newcommand{\Benu}{\begin{enumerate}}
\newcommand{\Eenu}{\end{enumerate}}
\newcommand{\Bi}{\begin{itemize}}
\newcommand{\Ei}{\end{itemize}}

\def\intslash{\rlap{\kern  .32em $\mspace {.5mu}\backslash$ }\int}
\def\qsl{{\rlap{\kern  .32em $\mspace {.5mu}\backslash$ }\int_{Q_x}}}

\def\ga{\gamma}

\def\cf{{\it cf}}

\def\loc{{\text{\rm loc}}}

\def\supp{{\text{\rm supp}}}

\def\noi{\noindent}

\def\meas{{\text{\rm meas}}}

\def\lc{\lesssim}
\def\gc{\gtrsim}

\def\eps{\varepsilon}

\def\ka{\kappa}
\def\vk{\varkappa}
            
\def\la{\lambda}             \def\La{\Lambda}
             
\def\si{\sigma}

\def\fJ{{\mathfrak {J}}}

\def\fM{{\mathfrak {M}}}

\def\fS{{\mathfrak {S}}}

\def\fz{{\mathfrak {z}}}

\def\bbC{{\mathbb {C}}}

\def\bbR{{\mathbb {R}}}

\def\bbZ{{\mathbb {Z}}}

\def\cC{{\mathcal {C}}}
\def\cD{{\mathcal {D}}}
\def\cE{{\mathcal {E}}}
\def\cF{{\mathcal {F}}}
\def\cG{{\mathcal {G}}}
\def\cH{{\mathcal {H}}}
\def\cI{{\mathcal {I}}}

\def\cK{{\mathcal {K}}}

\def\cO{{\mathcal {O}}}

\def\cT{{\mathcal {T}}}

\def\cV{{\mathcal {V}}}
\def\cW{{\mathcal {W}}}

\def\I{{\hbox{\bf I}}}

 at 10 true pt

\def\be#1{\begin{equation}\label{ #1}}
\def\endeq{\end{equation}}
\def\endal{\end{align}}
\def\bas{\begin{align*}}
\def\eas{\end{align*}}
\def\bi{\begin{itemize}}
\def\ei{\end{itemize}}

\def\eps{\varepsilon}

\def\emph#1{{\it #1}}
\def\textbf#1{{\bf #1}}
\usepackage{bbm}
\def\bbone{{\mathbbm 1}}

\theoremstyle{plain}
  \newtheorem{theorem}{Theorem}[section]

   \newtheorem{proposition}[theorem]{Proposition}
   
   \newtheorem{lemma}[theorem]{Lemma}
   \newtheorem{corollary}[theorem]{Corollary}

\theoremstyle{remark}

\theoremstyle{definition}

\begin{document}

\title[]
%[On square functions with  independent increments]
%{A characterization of Sobolev spaces}
{On square functions with   independent increments and Sobolev
spaces on the line}

\author[]{Juli\`a Cuf\'i   \ \ \   Artur Nicolau \ \ \  Andreas Seeger \ \ \  Joan Verdera}

\address{J. Cuf\'i, A. Nicolau, J. Verdera\\Departament de Matem\`atiques\\ Universitat Aut\`onoma de Barcelona\\ 08193 Bellaterra (Barcelona)\\ Catalonia}
\email{jcufi@mat.uab.cat} \email{artur@mat.uab.cat}
\email{jvm@mat.uab.cat}
%\address{A. Nicolau, Departament de Matem\`atiques\\ Universitat Aut\`onoma de Barcelona\\ 08193 Bellaterra (Barcelona)\\ Catalonia}

\address{A. Seeger, Department of Mathematics \\ University of Wisconsin \\480 Lincoln Drive\\ Madison, WI,
53706, USA} \email{seeger@math.wisc.edu}

%\address{J. Verdera, Departament de Matem\`atiques\\ Universitat Aut\`onoma de Barcelona\\ 08193 Bellaterra (Barcelona)\\ Catalonia}

\begin{abstract}  We prove a characterization of some $L^p$-Sobolev spaces
involving  the quadratic symmetrization of the Calder\'on
commutator kernel, which is related to a square function with
differences of difference quotients. An endpoint weak type estimate
is established for functions in homogeneous Hardy-Sobolev spaces
$\dot H^1_\alpha$.  We also use a  local version of this square
function to characterize
%,  up to sets of measure zero, 
pointwise differentiability 
%almost everywhere 
 for functions  in the Zygmund class.
\end{abstract}
\subjclass[2010]{} \keywords{}

\date\today

\maketitle

%\address{}

%\begin{thanks}\end{thanks}
%\begin{abstract}
%\end{abstract}

\section{Introduction}

In this paper we give a characterization of Sobolev spaces on the
real line by a square function which appears in some proofs  of the
$L^2$-boundedness of the first Calder\'on  commutator \cite{verdera}
and the Cauchy integral on a
 Lipschitz or chord arc curve \cite{mv}, \cite{verdera}. Moreover, a local version of this square function can be used to describe the set of points where a given function is pointwise differentiable.

Our square function acts on functions on the real line and involves
the
%differences  with different increments. First, let  the
difference of two difference quotients
%operator
with increments $s$ and $t$.
%$\Q_v=v^{-1}\Delta_v$
%\Be \cQ_vf(x)= \frac{\Delta_vf(x)}{v}=\frac{f(x+v)-f(x)}{v}.
%\Ee
Define \Be \label{squarefct} S_\alpha f(x) =
\Big(\iint_{\bbR\times\bbR}  \Big|\frac{f(x+s)-f(x)}{s} -
\frac{f(x+t)-f(x)}{t}  \Big|^2 \frac{ds\,dt}{|s-t|^{2\alpha}}
\Big)^{1/2}. \Ee 
This square function is a rough relative of the more standard Marcinkiewicz square function
associated with second differences,
\Be\label{Marcsq} 
G_\alpha f(x)= \Big(\int_0^\infty \frac{|f(x+2t)-2f(x+t)+f(x)|^2}{t^{1+2\alpha} }dt\Big)^{1/2},
\Ee which was introduced for $\alpha=1$ by Marcinkiewicz to investigate questions about pointwise
differentiability (see \cite{stein-sqf}). In \S\ref{ptsect} we prove  that for $\alpha\ge 0$ and $f\in L^2(\bbR)$ there is a pointwise majorization
\Be\label{Galpha-maj}G_{\alpha} f(x) \le C(\alpha) S_\alpha f(x).\Ee

We shall prove sharp results on  mapping properties of $S_\alpha$ when acting in $L^p$-Sobolev spaces.
Our starting point is the identity 
%\Be\label{L2identity} \|S_1f\|_2 =c \|f'\|_2\Ee
$\|S_1f\|_2 =c \|f'\|_2$,  proved in \cite{mv} by an application of Plancherel's theorem.
%This   identity   is the starting point of this paper in which 
%We aim for  a similar characterization of the space of functions  with
%distributional derivatives in $L^p$. 
We aim for an
%More generally, we ask for
analogous characterizations of other homogeneous Sobolev spaces
$\dot H^p_\alpha$, with $p\neq 2$ and suitable $\alpha$. It is proved  in \S\ref{necessary} that such a characterization is limited to the range $1/2<\alpha<3/2$.
Recall that, for
$1<p<\infty$, the (semi)-norm on $\dot H^p_\alpha$ is given by
$\|\cD^\alpha\! f\|_p$,  where $\cD^\alpha$ denotes  the Riesz
derivative operator of order $\alpha$; it is defined by $\widehat
{\cD^\alpha \!g}(\xi)=|\xi|^{\alpha}\widehat g(\xi)$,
%a priori
at least for  Schwartz functions whose Fourier transform is
compactly supported in $\bbR\setminus\{0\}$. Of course $\cD^\alpha$
is the inverse of the  Riesz potential operator
$I^\alpha=\cD^{-\alpha}$,
 given for $0<\alpha<1$ by $I^{\alpha}f=\gamma(\alpha)\cdot (f* |x|^{\alpha-1})$, with $\gamma(\alpha)$ a constant,
 and defined for other $\alpha $ by analytic continuation. The space
%defined by $\widehat {I^\alpha g}(\xi)=|\xi|^{-\alpha}\widehat g(\xi).$
%acting on tempered  distributions on the real line.
 $\dot H^p_\alpha$~
%For $1<p<\infty$  let  $\dot L^{p}_\alpha= \cD^{-\alpha} L^p$ be the homogeneous Sobolev (or Riesz  potential space)
consists  of all distributions which are Riesz potentials of order
$\alpha$ of $L^p$ functions.
%The norm is given by
%$\|f\|_{\dot L^p_\alpha}=\|\cD^{\alpha} f\|_p$.
See \cite{stein61}, \cite{stein-book}, \cite{triebel2},
\cite{triebel-hom} for more on  these spaces.

\begin{theorem}\label{sobchar}
Let $1<p<\infty$, $ 1/2<\alpha<3/2$. Then
%$$\|f\|_{\dot L^p_\alpha} \approx \|S_\alpha f\|_p\,.$$
$$\|\cD^{\alpha} \!f\|_{ L^p(\bbR)} \approx \|S_\alpha f\|_{L^p(\bbR)}\,.$$
Here the implicit constants depend only on $p$ and $\alpha$.
\end{theorem}
In contrast   we have  the larger range
$\max \{1/p-1/2,0\}< \alpha<2$
in
the known equivalence
$\|\cD^{\alpha} \!f\|_{p} \approx \|G_\alpha f\|_{p}$ 
for the Marcinkiewicz square function, see \cite{strichartz}.

The proof of Theorem \ref{sobchar}  is immediately reduced to the
equivalence \[ \|f\|_{ L^p(\bbR)} \approx \|S_\alpha (\cD^{-\alpha}
\!f)\|_{L^p(\bbR)}\,. \]
It is natural to ask whether this result extends to $p=1$ in the
sense of a characterization for the homogeneous Hardy-Sobolev spaces
$\dot H^1_\alpha$. 
The vector-valued operator  associated with $f\mapsto S_{\alpha}(\cD^{-\alpha} f)$  is not covered by the standard theory in 
\cite{bcp} but should be considered as a rough singular integral in the spirit of \cite{christ}.

We let $H^1$ stand for the usual Hardy space on the line, that is,  for the set of functions $f$ in $L^1$ such that the Hilbert transform of $f$ is also in $L^1$.
It turns out that the strong $H^1\to L^1$ boundedness of $f \rightarrow S_{\alpha}(\cD^{-\alpha})f $ fails; this is in contrast to a positive result for the Marcinkiewicz square function, namely 
$\|G_\alpha(\cD^{-\alpha} f)\|_1\lc \|f\|_{H^1}$ for $1/2<\alpha<2$. See e.g. \cite[\S3.5.3]{triebel2}.
The following    $H^1\to L^{1,\infty}$
endpoint result for $f\mapsto S^\alpha(\cD^{-\alpha}\!f)$
%for the Hardy space $H^1$
is optimal in the sense that $L^{1,\infty}$ cannot be replaced by a Lorentz space $L^{1,q}$ with $q<\infty$, see
\S\ref{hardyfailure}.

%cannot be replaced by an $H^1\to
%L^{1,q}$  result  for any $q<\infty$.

\begin{theorem}\label{sobhardy}
Let  $1/2<\alpha<3/2$. Then for all $f$ in the homogeneous
Hardy-Sobolev space $\dot H^1_\alpha$  and all $\la>0$,
%$$\|f\|_{\dot L^p_\alpha} \approx \|S_\alpha f\|_p\,.$$
$$\meas (\{x: S_\alpha f(x)>\lambda \} )\le C_\alpha \la^{-1}\|\cD^\alpha \!f\|_{H^1}.
$$
\end{theorem}

{The statement above for $\alpha =1$ can be restated in
terms of the first derivative, using the Hilbert transform.}
\begin{corollary}\label{sob1hardy}
(i) For $f\in L^p$, $1<p<\infty$, $\|S_1 f\|_p \approx \|f'\|_p$.

(ii) If $f'\in H^1$ then $\meas (\{x: S_1 f(x)>\lambda \} )\le C
\la^{-1}\| f' \|_{H^1}. $
\end{corollary}
In \S\ref{weaktype11failure} {we show} that the condition
$f'\in H^1$ in the second part of Corollary \ref{sob1hardy} cannot
be replaced by $f'\in L^1$, and formulate a related open question for the Riesz derivatives.

\medskip

%\subsubsection*{Pointwise differentiation}
We  shall also consider  a local version of the square
function~$S_{1}$ in order  to study pointwise differentiability.
Recall that a bounded function~$f\colon \mathbb{R}\to \mathbb{R}$ is
in the Zygmund class~$\Lambda_{*}$  (also known as the homogeneous
Besov space $\dot B^{\infty,\infty}_1$) if there exists a
constant~$c=c(f)>0$ such that $$\text{$|f(x+h)+f(x-h)-2f(x)|\le
c|h|$, $x,h\in\mathbb{R}$.}$$ The infimum of the constants~$c$
satisfying the above inequality is denoted by $\|f\|_{\Lambda_{*}}$.
Functions in the Zygmund class may be nowhere differentiable. For
example, the Weierstrass function $ f(x)=\sum^{\infty}_{n=1}b^{-n}
\cos (b^{n}x),
%\quad x\in \mathbb{R},
$ where $b>1$,  is nowhere differentiable  and belongs to $\La_*$.
{It turns out that a local version of $S_1$ characterizes
the almost everywhere differentiability of functions in the Zygmund
class, very much in the spirit of a classical theorem of Stein and
Zygmund \cite{stein-zyg} which uses a local version of the Marcinkiewicz square function $G_1$. Our result reads as follows.}
\begin{theorem}\label{ptwintro}
Let $f:\bbR\to \bbR$ belong to the Zygmund class $\Lambda_{*}$. Then
the set of points $x\in\mathbb{R}$ such that
$$
\iint_{|t|+|s|<1} \Big| \frac{f(x+s)-f(x)}{s}
-\frac{f(x+t)-f(x)}{t}\Big|^{2}\frac{ds\,dt}{|s-t|^{2}}<  \infty
$$
coincides,  except possibly for a set of Lebesgue measure zero, with
the set of points where $f$ is differentiable.
\end{theorem}
In view of the pointwise inequality $G_1f\lc S_1f$ 
%Our rough square function can be shown to dominate the square function in the Stein-Zygmund result (see Lemma  \ref{lem8.1}(b) below).
the main point of  Theorem \ref{ptwintro} is that almost everywhere the  pointwise differentiability for functions in the Zygmund class implies  the finiteness  of the rough square function. For general  functions in $L^2_\loc$  this implication fails, see \S\ref{ptfailure}.
Theorem \ref{ptwintro} will be obtained as  a simple consequence of a more general
% but also more technical
result formulated as  Theorem \ref{theo8.2}.

%As a consequence of our analysis we obtain the following result: let $f\in \Lambda_{*}$, then the set of points where $f$ is differentiable coincides, except possibly for a set of Lebesgue measure zero, with the set of points where the local version of~$S_{1}$ is finite.

{\it This paper.} In \S\ref{comm} we discuss the connection with
quadratic symmetrizations of Calder\'on commutators and with the
Cauchy integral. In \S\ref{ptsect} we 
prove a generalization of the pointwise majorization result \eqref{Galpha-maj}.
In \S\ref{necessary} we briefly discuss necessary
conditions for Theorems~\ref{sobchar} and~\ref{sobhardy}. In
\S\ref{lowerbounds} we prove the {lower bound  in
Theorem~\ref{sobchar}, namely, $\|\cD^\alpha\! f\|_p \lc \|S_\alpha
f\|_p$, $1<p<\infty$.} In \S\ref{L2sect} we discuss basic
decompositions of our operators and prove some refined $L^2$ bounds
that are crucial for the proofs of Theorems \ref{sobchar} and
\ref{sobhardy}. In \S\ref{hardy} we prove the endpoint Theorem
\ref{sobhardy}. In \S\ref{Lpsect} we quickly discuss various
approaches to  Theorem \ref{sobchar} via interpolation arguments. In
\S\ref{sec8} we state and prove the results on pointwise
differentiability.
%In \S\ref{opensect} we mention two open problems.
%by using an interpolation argument using Theorem \ref{sobhardy}, and a dual version of it. We also outline an argument on how to prove Theorem \ref{sobchar} without using $H^1\to L^{1,\infty}$ endpoint estimates.

% state a similar result for an adjoint operator and use these to deduce Theorem \ref{sobchar} by interpolation.
%An alternative (more straightforward) approach  to Theorem \ref{sobchar} which bypasses Theorem \ref{sobhardy} is outlined in \S\ref{alt-Lpsect}.
%In  \S\ref{upperbounds} we prove the $L^p$ bounds for the square-function  $S_\alpha$.

 %\medskip

 \section{Relation with Calder\'on commutators}\label{comm}
  For suitable  functions {$A:\bbR\to \bbC$ } consider  the first  Calder\'on commutator $\cC_A$
%$\cC_A$
whose Schwartz kernel $\cK_A$ is given by
 $$\cK_A(x,y)=p.v.  \frac{A(y)-A(x)}{(x-y)^2}.$$
Calder\'on \cite{calderon} proved  the $L^2(\bbR)$ boundedness of
$\cC_A$  for Lipschitz functions $A$; subsequently   many other
proofs were discovered. Here we are motivated by the proof in
\cite{verdera}  which uses a symmetrization technique based on
 the three term {\it quadratic symmetrization}
\begin{multline} \label{Symdef}\text{Sym}[\cK_A](x,y,z)\\
:=\cK_A(x,y)\cK_A(x,z))+ \cK_A(y,z)\cK_A(y,x)+\cK_A(z,x)\cK_A(z,y)
 {\color{red},}\end{multline}
which is well defined as a function  on $$G=\{(x,y,z)\in
\bbR\times\bbR\times\bbR: \, x\neq y, \, x\neq z,\, y\neq z\}.$$ We
have the following elementary but crucial identity (\cite{verdera}).
\begin{lemma} \label{SymKformulalemma}
For $(x,y,z)\in G$, \Be   \label{SymKformula}
 \text{\rm Sym}[\cK_A](x,y,z)=\frac1{(z-y)^2}\Big(\frac{A(y)-A(x)}{y-x}- \frac{A(z)-A(x)}{z-x}\Big)^2\,.
\Ee
\end{lemma}
\begin{proof} We use the notation $D_{ab}:=A(a)-A(b)$. For  all $(x,y,z)$ we compute
\begin{align*}
&(x-y)^2(x-z)^2(y-z)^2 \, \text{Sym}[\cK_A](x,y,z)
\\
=&(y-z)^2D_{xy}D_{xz}+(z-x)^2 D_{yz}D_{yx}+(x-y)^2 D_{zx}D_{zy}
\\
=&x^2D_{yz}(D_{yx}-D_{zx}) +y^2 D_{zx}(D_{zy}-D_{xy})+z^2
D_{xy}(D_{xz}-D_{yz})
\\&\qquad \qquad\qquad\qquad- 2yz D_{xy}D_{xz}-2zx D_{yz}D_{yx}-2xyD_{zx}D_{zy}
\end{align*}
and using $D_{ac}-D_{bc}=D_{ab} $ we see that $(x-y)^2(x-z)^2(y-z)^2
\, \text{Sym}[\cK_A](x,y,z)$ is equal to
\begin{subequations}
\Be
%\begin{align} & \notag\\
 \label{firstalg}
%&
x^2D^2_{yz}+y^2 D^2_{zx}+z^2D^2_{xy}-2xy
D_{xz}D_{yz}-2xzD_{xy}D_{zy}-2yzD_{yx}D_{zx}. \Ee
%\end{align}
Now  it turns out that this expression is also equal to
\Be\label{secondalg}
%\big( (z-x) (A(y)-A(x))-(y-x) (A(z)-A(x)) \big)^2\Ee
\big( (z-x) D_{yx}-(y-x) D_{zx}) \big)^2.\Ee
\end{subequations}
Indeed the last display equals
\begin{align*}
%&((z-x)D_{yx}-(y-x)D_{zx})^2\\
&(z-x)^2 D^2_{yx}+(y-x)^2 D_{zx}^2-2(z-x)(y-x) D_{yx}D_{zx}
\\&= z^2D_{yz}^2+y^2 D^2_{zx}-2zyD_{yx}D_{zx} +R
\end{align*}
where $R= -2xz(D^2_{yx}-D_{yx}D_{zx})-2xy(D^2_{zx}-D_{yx}D_{zx})-
x^2 (D_{yx}-D_{zx})^2$. Now use $D_{yx}-D_{zx}=D_{yz}$ and conclude
that \eqref{firstalg} and \eqref{secondalg} coincide. This  yields
the asserted formula.
%are both equal to\begin{align*}x^2\big( A(y)-A(z)\big)^2+y^2\big(A(z)-A(x)\big)^2+ z^2 \big( A(x)-A(y)\big)^2&\\- 2xy(A(x)-A(z))(A(y)-A(z))&\\- 2xz(A(x)-A(y))(A(z)-A(y))&\\- 2yz(A(y)-A(x))(A(z)-A(x)).& \qedhere\end{align*}
\end{proof}

Using  Lemma \ref{SymKformulalemma}  the result of Theorem
\ref{sobchar} can now be written in terms of the quadratic
symmetrization:
\begin{corollary} For $1<p<\infty$, $1/2<\alpha< 3/2$,
 $$\|\cD^\alpha\! A\|_p^p\approx  \int \Big(\iint \frac{ |\text{\rm Sym}[\cK_A](x,y,z)| }{|y-z|^{2\alpha-2}}dy\, dz\Big)^{p/2}\,dx.
$$
 \end{corollary}

As mentioned before, for $\alpha=1$, $p=2$ the equivalence of norms
becomes an identity (noted in \cite{mv}). Indeed \eqref{SymKformula}
and  a  Fourier transform calculation using Plancherel's theorem
yield \Be \label{L2identity}\iiint_{\bbR\times\bbR\times\bbR}
|\text{Sym}[\cK_A](x,y,z)| dx\, dy\, dz=c\int_\bbR |A'(x)|^2 dx. \Ee
This argument can also be applied to  the cases
%A similar argument applies to the cases
$1/2<\alpha<3/2$. In \cite{verdera} it is explained  how
\eqref{L2identity} can be used to prove the $L^2$ boundedness of
$\cC_A$ when $A$ is Lipschitz: {one checks the
assumptions of the $T(1)$ theorem of David and Journ\'e. In fact,
the $T(1)$ can be bypassed by a simple argument, which reduces
matters to the $H^1-BMO$ duality.}

Moreover,  in \cite{verdera} it is shown   that the action of the
Cauchy-integral operator for the Lipschitz graph on characteristic
functions of intervals is  majorized by the action of the  first
Calder\'on commutator. The argument uses crucially the concept of
{\it Menger curvature}  which is controlled by $\text{Sym}[K_A]$. To
be specific, the Menger curvature function associated to the graph
$\fz(x)=(x,A(x))$ is defined by  $(R(x,y,z))^{-1}$ where $R(x,y,z)$
is the radius of the circle through the points $\fz(x)$, $\fz(y)$,
$\fz(z)$
%(x,A(x))$, $(y,A(y))$, $(z, A(z)$
(the Menger curvature is zero if the three points lie on a straight
line). The crucial identity is
$$\frac{1}{R(x,y,z)}=\frac{4 \text{ area }(\cT(x,y,z))}{|\fz(y)-\fz(x)| |\fz(z)-\fz(y)| |\fz(x)-\fz(z)| }$$
where $\cT(x,y,z) $ is the  triangle  with vertices $\fz(x)$,
$\fz(y)$, $\fz(z)$.
%$(x,A(x))$, $(y,A(y))$, $(z, A(z))$.
The identity implies, after using $|\fz(a)-\fz(b)|\ge |a-b|$ and
\eqref{SymKformula},  the inequality \Be \label{Mengervssymm}
\frac{1}{R(x,y,z)} \le 2 \,\big|\text{Sym}[\cK_A](x,y,z)\big|^{1/2}.
%\frac2{|z-y|}\Big|\frac{A(y)-A(x)}{y-x}- \frac{A(z)-A(x)}{z-x}\Big|\,.
\Ee See \cite{mv}, \cite{verdera} for more on the proof of the $L^2$
boundedness of the Cauchy integral operator based on Menger
curvature. We do not emphasize Menger curvature in this paper since,
while the inequality \eqref{Mengervssymm}
 is efficient when  $A$ is a Lipschitz function (as then $|\fz(a)-\fz(b)\approx |a-b|$), it may be wasteful for the  Sobolev  classes of functions we are interested here.

\section{Comparison with  Marcinkiewicz type square functions}\label{ptsect}
Given $f\in L^{2} (\mathbb{R})$ and $m\in\mathbb{R}$, consider the
following  square functions defined for $x\in \bbR$ by
%\begin{align*}
%\langle f\rangle^{2}_{m}(x)
\Be \label{cGdef}\cG_{\alpha,m} f(x)=\Big( \int_{\mathbb{R}}\! \Big| \frac{f(x+mt)-f(x)}{mt} -\frac{f(x+t)-f(x)}{t}\Big|^2 \frac{dt}{|t|^{2\alpha-1}}\Big)^{1/2}.\Ee 
The square function $S_\alpha$ can be recovered from the $\cG_{\alpha,m}$ using the 
identity
\begin{equation}\label{eq1}
S_\alpha f(x)^2=2\int_{|m|>1}\cG_{\alpha,m} f(x)^2\frac{dm}{|m-1|^{2\alpha}}
\end{equation}
which follows by the change of variables $s=mt$ for $|s|\ge |t|$, and symmetry
considerations.
Conversely, the next lemma shows a pointwise domination of  $\cG_{\alpha,m} f$  in terms of   $S_\alpha f$, for every $m>1$.
For $m=2$ we have
$$
\cG_{\alpha,2} f(x)\,=\,\frac { G_\alpha f(x)}2$$
 for the  Marcinkiewicz square function $G_\alpha f$ and thus 
recover the pointwise inequality \eqref{Galpha-maj} stated in the introduction.

\begin{lemma}\label{lem3.1}
Let 
$\alpha\ge 0$. Then for $m>1$, there exists a constant $C_{\alpha,m}$
such that for all $f\in L^2(\bbR)$ 
$$\cG_{\alpha,m}f(x) \le C_{\alpha,m}  S_\alpha f(x), \text{  $x\in \bbR$.} $$
% \quad x\in \bbR, \quad f\in
%L^2(\bbR).$$
%where $C_m=2 \max\{1, m^{1-2\alpha}\} (\log m)^{-1}$.
%\langle f\rangle_{2}(x)\le CSf(x)$, $x\in \mathbb{R}$, for any $f\in L^{2}(\mathbb{R})$.
\end{lemma}

\begin{proof}
Fix $s\in\mathbb{R}$ and $m>1$.
We have
\begin{align*}
\cG_{\alpha,m}f(x)^2
%\langle f\rangle^{2}_{m} (x)
 \le 2\int_{\mathbb{R}} \Big| \frac{f(x+mt)-f(x)}{mt}-\frac{f(x+mt/s)-f(x)}{mt/s}\Big|^{2}\frac{dt}{|t|^{2\alpha-1}}&\\
\\+2\int_{\mathbb{R}}\Big|\frac{f(x+mt/s)-f(x)}{mt/s} -\frac{f(x+t)-f(x)}{t}\Big|^{2}\frac{dt}{|t|^{2\alpha-1}}&.
\end{align*}
The change of variable $mt/s\!=\!u$ shows that the first term is equal to
{$2 (s/m)^{2-2\alpha}\cG_{\alpha,s} f(x)^{2}$}. Hence
$$
\cG_{\alpha,m}f(x)^2\le 2 (s/m)^{2-2\alpha}\cG_{\alpha,s}f(x)^2+2\cG_{\alpha,m/s}f(x)^2.
%\langle f\rangle_{m}^{2}(x)\le 2 (\langle f\rangle^{2}_{s}(x)+\langle f\rangle^{2}_{m/s}(x)).
$$
Observe that the interval $[1,m]$ is invariant under the change of variable $s\mapsto m/s$. Integrating with respect to the measure $ds/s$  yields
\begin{align*}
 &\cG_{\alpha,m}f(x)^2\log m = \int_1^m \cG_{\alpha,m} f(x)^2 \frac{ds}{s}
 \\
 &\le 2 \int_1^m \big[(s/m)^{2-2\alpha}\cG_{\alpha,s}f(x)^2+\cG_{\alpha, m/s}f(x)^2\big] \frac{ds}{s}
 \\
 &= 2\int_1^m \cG_{\alpha,s} f(x)^2 \Big( \frac{s^{2-2\alpha}}{m^{2-2\alpha}}+1\Big) \frac{ds}{s} 
 \\
 &\le 2A_{\alpha,m} \int_1^m \cG_{\alpha,s} f(x)^2 \frac{ds}{(s-1)^{2\alpha}}
 \end{align*}
 where 
 $A_{\alpha,m}:=\sup_{1\le s\le m}  ((s/m)^{2-2\alpha}+1) s^{-1} (s-1)^{2\alpha}
 $  is clearly finite for $\alpha\ge 0$. Now by the identity \eqref{eq1} we obtain 
 \[\cG_{\alpha,m}f(x)\le \Big(\frac{A_{\alpha,m}}{\log m}\Big)^{1/2} S_\alpha f(x)\,.
 \qedhere\]
\end{proof}
\section{Necessary conditions}\label{necessary}
We show that our characterization fails to extend to the
Hardy-Sobolev spaces (corresponding to $p=1$) and that the
condition $1/2<\alpha<3/2$ in Theorem \ref{sobchar} is necessary.  In
what follows we use the notation \Be\label{difference} \Delta_s
f(x)= f(x+s)-f(x)\Ee for the difference  operator with increment
$s$.

The restriction $1/2<\alpha<3/2$ is known to be necessary in other similar contexts. For instance, that $\alpha < 3/2$ is a consequence of Proposition 2 in \cite{brezis}
applied to $ s^{-1} \Delta_s f(x)$. 
However we present below a direct argument for the case at hand.

\subsection{\it The condition $\alpha>1/2$} Suppose that $0<\alpha\le 1/2$.
Consider $f\in C^\infty_c$ with vanishing moments up to order $2$,
with the property that $f(x)=1$ for $x\in [0,1]$ and $f$ is
supported in $(-3/2,3/2)$. Then $f\in \dot H^p_\alpha$ for $p\ge 1$. Notice that
$$s^{-1}\Delta_s f(x)-t^{-1}\Delta_tf (x) = -s^{-1}+ t^{-1} \text{
for  $x\in [0,1]$, $t\in [\tfrac 32,2]$, $s>4$,}$$ and thus, for
$x\in [0,1]$,
$$S_\alpha f(x)\ge \frac{1}{4} \Big(\int_{s=4}^\infty |s-2|^{-2\alpha} ds\Big)^{1/2}=\infty,\, \text{ if } \alpha\le 1/2\,.$$
Thus we need to have $\alpha>1/2$.

\subsection{\it The condition $\alpha<3/2$}
Let $f\in C^\infty_c$ with vanishing moments up to  order $2$
%{\color{red} (Be more precise here ! How many moments are required?)} 
{and satisfying} $f(x)=x^2$ for $|x|\le 4$. 
As above, $f\in \dot H^p_\alpha$ for $p\ge 1$.
Now
$$s^{-1}\Delta_s f(x)-t^{-1} \Delta_t f(x)
= \int_0^1 f'(x+us)-f'(x+ut) du$$
%\begin{align*} \cQ_s f(x)-\cQ_t f(x) &= \int_0^1 f'(x+us)-f'(x+ut) du\\&=(s-t) \int_0^1\int_0^1 f'' (x+ut+ vu(s-t)) \,udu\, dv\end{align*}
so that $s^{-1}\Delta_s f(x)-t^{-1}\Delta_t f(x) = s-t$ if $|x|\le
1$, $|s+t|\le 1$, $|s-t|\le 1$, and we get
$$S_\alpha f(x)\ge \Big(\iint_{\substack{|s-t|\le 1\\|s+t|\le 1} }|s-t|^{2-2\alpha} \, ds\, dt\Big)^{1/2}, \text{ for } |x|\le 1.
$$
Hence, if $\alpha\ge 3/2$ then $S_\alpha f(x)=\infty$ for $|x|\le 1$
which shows the necessity of the condition $\alpha<3/2$.

\subsection{\it Failure of the strong type Hardy space bound}\label{hardyfailure}
We show that for functions in the homogeneous  Hardy-Sobolev spaces
$\dot H^1_\alpha$ the square-function $S_\alpha f$ may fail to be in $L^1$, or even in any Lorentz space
$L^{1,q}$ with $q<\infty$.

Let  $f$ be an odd  smooth  function with compact support in
$(-2,2)$   such that
%$\int f(x) x^j  dx=0$, for $j=0,1,2,$ and such that
%$\int f(x) dx=0$, and such that
$f(y)=1$ for $y\in [1/2,1]$.  Using  dyadic frequency decompositions
one can show that $\cD^\alpha f \in H^1(\bbR)$ for $\alpha\ge 0$.
%$0\le \alpha<2$.
 Let $x>2$.
We then  have
\begin{align*}
{\Delta_s f(x)}&=1\quad \text{ if } -x+1/2\le s\le -x+1,
\\
 {\Delta_t f(x)}&=-1 \quad \text{ if } -x-1\le t\le -x-1/2.
 \end{align*}
% for $x+v \in [1/2,1]$ we have   and for $x+w\in [-1,-1/2]$ we have $\cQ_w f(x)= -w^{-1}$.
Hence, by \eqref{squarefct} we get for  $x>2$,
\begin{align*}
S_\alpha f(x)&\ge \Big(\int_{s=-x+1/2}^{-x+1}\int_{t=-x-1}^{-x-1/2}
\frac{|s^{-1}+t^{-1}|^2}{|s-t|^{2\alpha}} dvdw\Big)^{1/2}
%\\&= \Big(\int_{v=-x+1/2}^{-x+1}\int_{w=-x-1}^{-x-1/2}\frac{1}{(v-w)^2}\Big(\frac 1v+\frac 1w\Big)^2  dvdw\Big)^{1/2}
%\\&
\ge \frac 1{2(x-1)}\,
\end{align*} and thus   $S_\alpha  f\notin L^{1,q}(\bbR)$ for $q<\infty$.
%\footnote{It  would be interesting to figure out whether the weak type inequality
% $\|S_\alpha (I^{\alpha}  f)\|_{L^{1,\infty}} \lc \| f\|_{H^1}$ \, or
% $\|S_1 f\|_{L^{1,\infty} } \lc \|f'\|_{H^1}$ holds.}

\subsection{\it Failure of a weak type $(1,1)$ bound} \label{weaktype11failure}
%Let $\alpha=1$.
%We show that in part (ii)
We prove the statement given after  Corollary \ref{sob1hardy}
%Let $\alpha=1$.
and  show that there is a sequence of functions $f_j$ such that  $\|f_j'\|_1
=O(1)$ and the $S_1 f_j$ are unbounded in $L^{1,\infty}$.

Define $f_j(x)=0$ for $x\le 0$, $f_j(x)=jx$, for $0\le x\le j^{-1}$
and $f_j(x)=1$ for $x> j^{-1}$ so that $f_j$ is a regularized
version of the Heaviside function. We have
$f_j'=j\bbone_{[0,j^{-1}]}$ so that $\|f_j'\|_1=1$.

Let now $-3/4\le x\le -1/2$. Then $f_j(x)=0$ and $f_j(x+s)=1$ if
$-x+1/j\le s\le 1$, moreover $f_j(x+t)=0$ for $j^{-1} \le t\le -x$.
We thus get, for $j\ge 100$,
\begin{align*} S_1f(x) &\ge \Big(\int_{s=-x+j^{-1} }^1s^{-2} \int_{1/j}^{-x} (s-t)^{-2} dt\, ds\Big)^{1/2}
\\ &\ge c \Big( \int_{-x+j^{-1}}^1 \big((s+x)^{-1} - (s-j^{-1})^{-1}\big)\,ds \Big)^{1/2}
\ge c' \big( \log j -C\big)^{1/2}.
\end{align*}
Hence for large $j$ and small $c$
$$\meas \{x: S_1f_j(x)\ge c\sqrt{\log j}\} \ge 1/4$$ which shows
%that  the $L^{1,\infty}$ quasinorm satisfies
$\|S_1 f_j\|_{L^{1,\infty}} / \|f_j'\|_1\gc \sqrt{\log j}$.

\smallskip

\noi {\it Open problem:} It would be interesting to explore  what happens
if the ordinary derivative $f'$ is replaced by the Riesz-derivative
$\cD^1 \!f$. More generally, does the weak type $(1,1)$ inequality $\|S_\alpha
f\|_{L^{1,\infty}}\lc \|\cD^\alpha\! f\|_{L^1}$ hold for
$1/2<\alpha<3/2$?

\section{$L^p$ converse estimates}\label{lowerbounds}
It is our objective to prove the converse estimate \Be
\label{converse} \|f\|_p \le C_{\alpha,p} \|S_\alpha
(\cD^{-\alpha}\!f)\|_p \Ee for $1<p<\infty$. There is no restriction
on $\alpha$ in this part of the proof.

\bigskip

First consider the function $\rho(s)= s^{-1} (e^{is}-1) $ and
observe that
%{\color{red}$$\rho(s)-\rho(t) = \frac{1}{2}(s-t) + E(s,t)$$}
$$\rho(s)-\rho(t) = \frac{1}{2}(t-s) + E(s,t)$$
where $|E(s,t)|\le C (|s|+|t|)|s-t|$ for $|s|, |t|\le 1$. Let
$$R_\eps=\{(s,t): \eps<s<2\eps,\, \eps/10<|s-t|<\eps/5\};$$ then for
sufficiently small $\eps>0$ the function
$$m(\xi)= \frac 1{|R_\eps|} \iint_{R_\eps} \Big( \frac{e^{is\xi}-1}{s}-\frac{e^{it\xi}-1}{t}\Big) |s-t|^{-\alpha} \, ds dt $$
is smooth on $[-4,-1/4]\cup [1/4,4]$,  and moreover
$$|m(\xi)| \ge C_{\eps,\alpha}>0, \quad 1/4\le|\xi|\le 4.$$

Let $\varphi$ be supported in $(1/2,2) $ such that $\sum_{k\in \bbZ}
\varphi(2^{-k}(|\xi|) )=1$  for all $\xi\neq 0$. Let $\widetilde
\varphi \in C^{\infty}$ be supported in $(1/4,4) $ such that
$\widetilde \varphi=1$ on $[1/2,2]$. Then $\xi \mapsto \widetilde
\varphi(|\xi|)/ m(\xi)$ is a $C^\infty_c$ function with support in
$\{\xi:1/4<|\xi|<4\}$.

Define  three operators $L_k$, $M_k$,  $\widetilde L_{k,\alpha}$ by
\begin{align*}
\widehat {L_k f}(\xi) &= \varphi(2^{-k}|\xi|) \widehat f(\xi)\,,
\\\
\widehat {M_k f}(\xi) &= \frac{\widetilde \varphi(2^{-k}|\xi|) }{
m(2^{-k}\xi)}  \widehat f(\xi)\,,
\\
\widehat {\widetilde L_{k,\alpha} f}(\xi) &= \widetilde
\varphi(2^{-k}|\xi|) (2^{-k}|\xi|)^\alpha  \widehat f(\xi)\,.
\end{align*}
These convolution operators make sense for Hilbert-space valued
functions.

Below we shall use the following
\begin{lemma} \label{vv} Let $1<p<\infty$.

(i) For $\{f_k\} \in L^p(\ell^2)$
$$\Big \|\sum_{k\in \bbZ} L_k f_k \Big\|_{p} \le C_p \Big\|\Big(\sum_{k\in \bbZ}|f_k|^2\Big)^{1/2}\Big\|_p\,.$$

(ii) For $\{f_k\} \in L^p(\ell^2)$
$$\Big \|\Big(\sum_{k\in \bbZ} |M_k f_k|^2\Big)^{1/2}  \Big\|_{p} \le C_p \Big\|\Big(\sum_{k\in \bbZ}|f_k|^2\Big)^{1/2}\Big\|_p\,.$$

(iii)  Let $\cH$ be a Hilbert space. For $F\in L^p(\cH)$ we have
$$\Big \|\Big(\sum_{k\in \bbZ} |\widetilde L_{k,\alpha}
F|_{\cH}^2\Big)^{1/2}  \Big\|_{p} \le C_{p,\alpha}
\|F\|_{L^p(\cH)}\,.$$
\end{lemma}
\begin{proof}
These are straightforward  applications of the standard  theory of
singular convolution operators  for Hilbert-space valued functions,
see \cite{bcp}, \cite{stein-book}.\end{proof}

\begin{proof}[Proof of \eqref{converse}]
Define $L_k$ as above and $\widetilde L_k$ similarly, with $\varphi$
replaced by $\widetilde \varphi$. We then have
$$ f= \sum_{k\in \bbZ} L_k f = \sum_{k\in \bbZ} L_k \widetilde L_k \widetilde L_k f
= \sum_{k\in \bbZ} L_k   M_k \widetilde L_k \cF^{-1} [
m(2^{-k}\cdot) \widehat f].
$$
By (i), (ii)  of Lemma \eqref{vv} we have
\begin{align*}
\|f\|_p &\lc \Big\|\Big(\sum_{k\in \bbZ}\big |M_k \widetilde L_k
\cF^{-1} [ m(2^{-k}\cdot) \widehat f]\big|^2\Big)^{1/2} \Big\|_p
\\
&\le \Big\| \Big(\sum_k \Big| \frac{1}{|R_\eps|} \iint_{R_\eps}
T^{s,t}_k\! f\, ds\, dt\Big|^2 \Big)^{1/2} \Big\|_p
\end{align*}
where $T^{s,t}_k$ is defined by
\begin{multline*}
\widehat {T^{s,t}_k \!f}(\xi) = \widetilde \varphi(2^{-k} |\xi|)
\Big( \frac{ e^{i s (2^{-k}\xi)}-1}{s} -\frac{ e^{i t
(2^{-k}\xi)}-1}{t} \Big) |s-t|^{-\alpha} \widehat f(\xi)
\\
\,\, =2^{-k} \widetilde \varphi(2^{-k} |\xi |) (2^{-k} |\xi|)^\alpha
\Big( \frac{ e^{i (2^{-k} s) \xi}-1}{2^{-k}s} -\frac{ e^{i (2^{-k}
t) \xi}-1}{2^{-k} t} \Big) \frac{\widehat {\cD^{-\alpha}\!
f}(\xi)}{(2^{-k} |s-t|)^{\alpha} } \,.
\end{multline*}
We apply the Cauchy-Schwarz inequality on $R_\eps$ and get
$$
\|f\|_p \le \Big\| \Big(\sum_k \frac{1}{|R_\eps|} \iint_{R_\eps}
\big|T^{s,t}_k\!f\big|^2 ds\,dt\Big)^{1/2} \Big\|_p \,.
$$
Change variables $s=2^k v $, $t=2^k w$  so that the last inequality
becomes
\begin{multline*} \|f\|_p\\
\lc  \Big\| \Big(\sum_k \frac{1}{|R_\eps|} \iint_{2^{-k} R_\eps}
\Big| \frac{\widetilde L_{k,\alpha} \,[v^{-1}\Delta_v(\cD^{-\alpha}
\!f)-w^{-1}\Delta_w(\cD^{-\alpha} \!f)]}{|v-w|^\alpha}  \Big|^2dv\,
dw\Big)^{1/2} \Big\|_p \,.
\end{multline*}
We replace for each $k$ the domain of integration $2^{-k} R_\eps$ by
the entire $\bbR\times\bbR$ and then apply part (iii) of Lemma
\ref{vv} (with the Hilbert space $\cH=L^2(\bbR\times\bbR)$). We thus
see that $\|f\|_p$ is bounded by (a constant times)
%The result is
\begin{align*}
%\|f\|_p&\lc
&|R_\eps|^{-1/2}\Big\| \Big(\sum_k \iint_{\bbR\times
\bbR}\frac{\big|\widetilde L_{k,\alpha} [
v^{-1}\Delta_v(\cD^{-\alpha }\!f) -w^{-1} \Delta_w(\cD^{-\alpha
}\!f)] \big|^2} {|v-w|^{2\alpha} }dv\, dw\Big)^{1/2} \Big\|_p
\\
& \le C_{\eps,p,\alpha} \Big\| \Big(\iint_{\bbR\times \bbR}
\frac{\big| v^{-1}\Delta_v(\cD^{-\alpha} \!f)
-w^{-1}\Delta_w(\cD^{-\alpha} \!f)\big|^2} {|v-w|^{2\alpha} }dv\,
dw\Big)^{1/2} \Big\|_p
\end{align*}
which completes the proof of \eqref{converse}.
\end{proof}

\bigskip

\section{$L^2$ bounds} \label{L2sect}
As mentioned before the equivalence
\begin{equation}\label{L2equiv}
\|S_{\alpha}(\cD^{-\alpha}\!f)\|_2= c_\alpha\|f\|_2
\end{equation}
has been proved for $\alpha=1$ in \cite{mv}; a   straightforward
modification of the proof also applies to the case $\alpha\in
(1/2,3/2)$. In this section we further break up
$S_{\alpha}(\cD^{-\alpha}f)$ and obtain improved bounds for the
pieces, which are useful for the proof of Theorems \ref{sobchar} and
\ref{sobhardy}.

%Let $\cQ_s$ denote the difference quotient operator with increment $v$,
%\label{upperbounds} \Be \cQ_sf(x)= \frac{\Delta_sf(x)}{v}=\frac{f(x+s)-f(x)}{v}.\Ee
Let $\cH$ be the Hilbert space of square-integrable functions on
$\bbR\times\bbR$. Fix $\alpha\in (1/2,3/2)$. We define a convolution
operator $T$ mapping Schwartz functions on $\bbR$ to $\cH$-valued
functions on the real line, by \Be \label{Tdef} Tf(x,s,t)=
\frac{s^{-1}\Delta_s\cD^{-\alpha}\! f(x)-t^{-1}\Delta_t
\cD^{-\alpha} \!f(x) }{|s-t|^{\alpha}}\text{ for $|s|\ge |t|$, } \Ee
and $$Tf(x,s,t)=0 \text{ for  $|s|<|t|$.}$$ The inequality
$\|S_\alpha (\cD^{-\alpha }f)\|_p\lc \|f\|_p$ holds for all Schwartz
functions $f$ if and only if
 $T$ maps $L^p(\bbR)$ to $L^p(\bbR;\cH)$.
 %By symmetry considerations it suffices to estimate $Tf(x,v,w) \bbone_{|v|\ge |w|}(v,w)$.
For the estimates below we may assume that $f$ is a Schwartz
function whose Fourier transform is compactly supported in
$\bbR\setminus\{0\}$.

We introduce finer decompositions by dividing up the $(s,t)$
parameter set. {For} $n,l\in \bbZ$,  $l \le n-2$, {set}
%$$\cU_k=\{(v,w): |v|\le 2^{-k+2}, \, |w|\le 2^{-k+2}, |w|\le |v|\}$$
%and for $n=2,3,\dots$, $$\cV^{n,0}_k=\{(v,w): 2^{-k+n}< |v|\le 2^{-k+n+1}, \, 2^{-k+n-2}< |w|\le |v|,\, |v-w|\le 2^{-k}\}$$ and
\Be\begin{aligned}\label{cVdef} \cV^{n,l}_k&=\{(s,t): 2^{-k+n}<
|s|\le 2^{-k+n+1}, \, 2^{-k+n-2}<|t|\le |s|,\,
\\&2^{-k+l-1}<|s-t|\le 2^{-k+l}\},
%\quad l=1,2,\dots\,.
\end{aligned}
\Ee and note that  \Be \cV^{n,l}_k=\emptyset \text{  for }l\ge n+3.\Ee

%Also, again for $n=2,3,\dots$,  let$$\cW_k^{n,0}=\{(v,w): 2^{-k+n}< |v|\le 2^{-k+n+1}, \,  |w|\le 2^{-k}\}$$ and, for $ \ell=1,\dots, n-2$,
Also for $\ell \in \bbZ$, $\ell\le n-2$, let \Be\label{cWdef}
\cW_k^{n,\ell}=\{(s,t): 2^{-k+n}< |s|\le 2^{-k+n+1}, \,
2^{-k+\ell-1}< |t|\le 2^{-k+\ell}\}. \Ee Then, for every $k\in
\bbZ$,
$$  \sum_{n\in \bbZ} \Big(\sum_{l=-\infty}^{n+2} \bbone_{\cV_k^{n,l}}(s,t)+\sum_{\ell=-\infty}^{n-2}
\bbone_{\cW_k^{n.\ell}}(s,t) \Big)= \begin{cases} 1&\text{ if
$|t|\le |s|$,}
\\
0&\text{ if $|t|> |s| $}.
\end{cases}
%\bbone_{|w|\le |v|}(v,w)\,.
$$
We also observe \Be\label{setscaling}
%\cU_k=2^{-k}\cU_0,\,\,
\cV_k^{n,l}=2^{-k} \cV_0^{n,l},\,\,\,
\cW_k^{n,\ell}=2^{-k}\cW_0^{n,\ell}. \Ee

%The upper bound in Theorem \ref{sobchar} follows from summing the estimates in the following theorem.

In what follows we denote by $\psi$  {a real} valued Schwartz
function so that $\widehat\psi(\xi)\neq 0$ for $\tfrac 14\le|\xi|
\le 4$ and $\widehat \psi$ vanishes to order $100$ at the origin. We
may choose $\psi$ so that
%In addition we shall assume
\Be\label{support-assu} \supp(\psi) \subset \{x: |x|\le 1/2\},\Ee We
remark that this assumption is not needed in the present section,
nor in the proof of Theorem \ref{sobchar}  discussed in
\S\ref{altLpproof}. However it is quite convenient in the proof of
the endpoint bound of Theorem \ref{sobhardy}.

Set $\psi_k= 2^{k}\psi(2^{k}\cdot)$. Define an operator $P_k$ by \Be
\label {Pk} P_k f=\psi_k*f.\Ee We introduce a decomposition of the
operator $T$. Let $\varphi\in C^\infty_c$ supported in $\{\xi:
1/2<|\xi|<2\}$ so that $\sum_{k\in \bbZ} \varphi(2^{-k}\xi)=1$ for
all $\xi\neq 0$. We then decompose
\begin{multline*}\widehat {Tf}(\xi, s,t)\\= \sum_{k\in\bbZ}\widehat \psi (2^{-k}\xi)^2\frac{\varphi(2^{-k}\xi)}{(2^{-k}|\xi|)^\alpha(\widehat\psi(2^{-k}\xi))^2}
\frac{1}{2^{k\alpha}|s-t|^\alpha} \Big(\frac{e^{i\xi
s}-1}{s}-\frac{e^{i\xi t}-1}{t}\Big)\widehat f(\xi)
\end{multline*}
and hence \Be\label{Tdec} Tf(x,s,t)= \sum_{k\in \bbZ} P_k T_k L_k
f(x,s,t)\Ee where \Be\label{Lk}\widehat{ L_k f}(\xi) =
\frac{\varphi(2^{-k}\xi)}{(2^{-k}|\xi|)^\alpha(\widehat\psi(2^{-k}\xi))^2}\widehat
f(\xi)\Ee and \Be \label{Tk}T_kf(x,s,t)=
\frac{1}{2^{k\alpha}|s-t|^\alpha}\Big( \frac{P_k f(x+s)-P_k f(x)}{s}
- \frac{P_kf(x+t)-P_kf(x)}{t}\Big)\, \Ee when $|t|\le |s|$ (and $T_k
f(x,s,t)=0$ otherwise). We also set, for $|t|<|s|$,
\begin{subequations} \begin{align}\label{Tk1}
T_{k,1}f(x,s,t)&= \frac{1}{2^{k\alpha}|s-t|^\alpha}\frac{P_k
f(x+s)-P_k f(x)}{s},
\\  \label{Tk2} T_{k,2}f(x,s,t)&=
\frac{1}{2^{k\alpha}|s-t|^\alpha} \frac{P_kf(x+t)-P_kf(x)}{t}\,
\end{align}
\end{subequations}
so that $T_k=T_{k,1}-T_{k,2}$.

%In this  {\color{red} section}
We shall  repeatedly use the following scaling lemma.
\begin{lemma}\label{scalinglemma}  Let $g$ be a Schwartz function on $\bbR$.
For $k\in \bbZ$ and $\Omega\subset \bbR^2$, \Be \Big(\iint_\Omega
|T_k g(x, s,t) |^2 ds\, dt\Big)^{1/2} = \Big(\iint_{2^{k}\Omega}
|T_0 [g(2^{-k}\cdot)] (2^k x, v,w) |^2 dv\, dw\Big)^{1/2} \Ee
\end{lemma}
\begin{proof}
The left hand side is equal to
\begin{align*}
&\Big(\iint_\Omega \Big|\frac{P_k g(x+s)-P_kg(x)}{s}-\frac{P_k
g(x+t)-P_kg(x)}{t} \Big|^2 \frac{ds\,
dt}{2^{2k\alpha}|s-t|^{2\alpha}}\Big)^{1/2}
\\=&
\Big(\iint_{2^{k}\Omega }\Big|\frac{P_k
g(x\!+\!2^{-k}v)-P_kg(x)}{v}-\frac{P_k g(x\!+\!2^{-k}w)-P_kg(x)}{w}
\Big|^2 \frac{dv\, dw}{|v-w|^{2\alpha}}\Big)^{1/2}
\end{align*}
and the assertion follows from $P_k g(x)=
P_0[g(2^{-k}\cdot)](2^kx)$.
\end{proof}

Our  {proof} of $L^2$ boundedness  {involves} the following
elementary estimates.
\begin{lemma} \label{sigmatauintegrallemma}
Let $a\ge b>0$. Then

(i)
%\Be \label{UVtype}
$$\Big(\iint\limits_{\substack{|\sigma|\approx |\tau| \approx a\\
|\si-\tau|\approx b} }\Big|
\frac{e^{i\si}-1}{\si}-\frac{e^{i\tau}-1}{\tau}\Big|^2 \frac{d\si
d\tau}{|\si-\tau|^{2\alpha}}\Big)^{1/2} \lc \begin{cases} a^{\frac
12}b^{\frac 32-\alpha}&\text{ if } a\le 1,
\\
a^{-\frac 12}b^{\frac 32-\alpha}&\text{ if } b\le 1\le a,
\\
a^{-\frac 12}b^{\frac 12-\alpha}&\text{ if } 1\le b\le a.
\end{cases}
$$
%\Ee

(ii) Let $a\ge \ga$ and $\Omega_{a,\ga} :=\{ (\sigma,\tau):
|\tau|\le |\si|/2,\, |\tau|\approx \ga, \, |\si|\approx a \}$. Then
for $a\ge 1$
\begin{align*}
&\Big(\iint_{\Omega_{a,\ga}}\Big|\frac{e^{i\si}-1}{\si}\Big|^2
\frac{d\sigma d\tau}{|\sigma-\tau|^{2\alpha}}\Big)^{1/2} \lc
\ga^{\frac 12}a^{-\alpha} \min \{a^{\frac 12}, a^{-\frac 12}\}\,,
\\
&\Big(\iint_{\Omega_{a,\ga}} \Big|\frac{e^{i\tau}-1}{\tau}\Big|^2
\frac{d\sigma d\tau}{|\sigma-\tau|^{2\alpha}}\Big)^{1/2} \lc
a^{\frac 12-\alpha} \min \{\ga^{\frac 12}, \ga^{-\frac 12}\}\,.
\end{align*}
Moreover, if $a\le 1$ then
$$\Big(\iint\limits_{\Omega_{a,\ga}}\Big|
\frac{e^{i\si}-1}{\si}-\frac{e^{i\tau}-1}{\tau}\Big|^2 \frac{d\si
d\tau}{|\si-\tau|^{2\alpha}}\Big)^{1/2} \lc \gamma^{\frac 12}
a^{\frac 32-\alpha}.
$$
\end{lemma}

\begin{proof}This follows  {readily from}
%$$\Big|\frac{e^{i\sigma}-1}{\sigma} - \frac{e^{i\tau}-1}{\tau} \Big|\le C |\sigma-\tau|, \quad |\sigma|, \,|\tau|\le 10.
%$$ and
\Be \label{elemest1}\Big|\frac{e^{i\sigma}-1}{\sigma} -
\frac{e^{i\tau}-1}{\tau} \Big|\le C
\frac{|\sigma-\tau|}{1+|\sigma|+|\tau|} ,  \quad|\sigma-\tau|\le 10.
\Ee and \Be\label{elemest2} \Big|\frac{e^{ir}-1}{r}  \Big|\le C
\min\{1, r^{-1}\}. \qedhere \Ee
\end{proof}

\begin{proposition}
\label{L2est} Let $1/2<\alpha<3/2$ and let $\cV_k^{n,l}$, $\cW_k^{n,\ell}$ be as in \eqref{cVdef}, \eqref{cWdef}. Then the following estimates
hold.
%, provided that $$\Big\|\Big(\sum_k|f_k|^2\Big)^{1/2} \Big\|_2\le 1.$$
\begin{subequations}
%(i)  \Be\label{UL2}\Big\| \Big( \iint\Big| \sum_{\substack {k\in \bbZ:\\ (s,t)\in  \cU_k}} P_k T_k f_k(\cdot,s,t)\Big|^2ds dt\Big)^{1/2}\Big\|_2 \\ \lc \Ee

%(ii) For $n\ge 2$, $\alpha<3/2$,
%\begin{multline}\label{V0L2} \Big\| \Big( \iint\Big| \sum_{\substack {k\in \bbZ:\\ (s,t)\in  \cV_k^{n,0}}} P_k T_k f_k(\cdot,s,t)\Big|^2ds dt\Big)^{1/2}\Big\|_2 \\ \lc 2^{-n/2}\Big\|\Big(\sum_k|f_k|^2\Big)^{1/2} \Big\|_2.\end{multline}
%For $n\ge 2$, $l\ge 0$,
(i)\begin{multline}\label{VlL2} \Big\| \Big( \iint \Big|
\sum_{\substack {k\in \bbZ:\\ (s,t)\in  \cV_k^{n,l}}} P_k T_k
f_k(\cdot,s,t)\Big|^2ds \,dt\Big)^{1/2}\Big\|_2 \lc c_{n,l}
\Big\|\Big(\sum_k|f_k|^2\Big)^{1/2} \Big\|_2\,,
\\ \text{ with }
c_{n,l}=\begin{cases} 2^{\frac n2} 2^{l(\frac 32-\alpha)} &\text{ if
} n\le 0, \, l\le n+2,
\\
2^{-n/2} 2^{l(\frac 32-\alpha)}&\text{ if } n\ge 0, \, l\le 2,
\\
2^{-n/2} 2^{-l(\alpha-\frac 12)}&\text{ if } n\ge 0, \, 0\le l\le
n+2,
\\
0 &\text{ if } l>n+2.
\end{cases}
%\Big\|\Big(\sum_k|f_k|^2\Big)^{1/2} \Big\|_2.
\end{multline}

(ii) Let $n\ge 0$ and $\ell\le n-2$. Then
\begin{multline}\label{WellL2-1}\Big\| \Big( \iint \Big|
\sum_{\substack {k\in \bbZ:\\ (s,t)\in  \cW_k^{n,\ell}}} P_k T_{k,1}
f_k(\cdot,s,t)\Big|^2ds \,dt\Big)^{1/2}\Big\|_2 \\ \lc 2^{\ell/2}
2^{-n\alpha} \min\{ 2^{-n/2}, 2^{n/2}\}
\Big\|\Big(\sum_k|f_k|^2\Big)^{1/2} \Big\|_2
\end{multline}
and
\begin{multline}\label{WellL2-2}\Big\| \Big( \iint
\Big| \sum_{\substack {k\in \bbZ:\\ (s,t)\in  \cW_k^{n,\ell}}} P_k
T_{k,2} f_k(\cdot,s,t)\Big|^2ds \,dt\Big)^{1/2}\Big\|_2 \\ \lc
2^{-n(\alpha-\frac 12)} \min\{2^{-\ell/2} , 2^{\ell/2}\}
\Big\|\Big(\sum_k|f_k|^2\Big)^{1/2} \Big\|_2\,.
\end{multline}
Moreover, for $n\le 0$,
\begin{multline}\label{WellL2}\Big\| \Big( \iint
\Big| \sum_{\substack {k\in \bbZ:\\ (s,t)\in  \cW_k^{n,\ell}}} P_k
T_{k} f_k(\cdot,s,t)\Big|^2ds\, dt\Big)^{1/2}\Big\|_2 \\ \lc
2^{n(\frac 32-\alpha)}  2^{\ell/2}
\Big\|\Big(\sum_k|f_k|^2\Big)^{1/2} \Big\|_2\,.
\end{multline}
\end{subequations}
\end{proposition}
\begin{proof}
 {Note} that for fixed $n,l$ the sets $\cV^{n,l}_k$, $k\!\in \bbZ$ are disjoint and,  similarly, for fixed
$n,\ell$ the sets $\cW^{n,\ell}_k$, $k\in \bbZ$ are disjoint. Hence
$\iint\!|\sum_k\! \cdots|^2 ds dt= \iint\sum_k| \cdots|^2 ds dt .$
%We use the almost orthogonality property of the $P_k$  to see that for $F_k\in L^2(\cH)$
%\Be \label{Pkorth}
%\Big\|\sum_{k\in \bbZ} P_kF_k\Big\|_{L^2(\cH)} \lc \Big(\sum_k \iint \big\|F_k(\cdot, s,t) \big\|_2^2 ds\, dt\Big)^{1/2}.
%\Ee
Thus, if one  then interchanges sums and integrals and uses the
uniform $L^2$ boundedness of the operators $P_k$ one can  reduce the
proofs to showing uniform estimates for the individual operators
$T_k$
 (or $T_{k,1}$, $T_{k,2}$), involving the sets
%$\cU_k$,
 $\cV^{n,l}_k$, $\cW^{m,\ell}_k$.  {By Lemma}~\ref{scalinglemma} this is reduced
to use estimates for the operator $T_0$ (or $T_{0,1}$, $T_{0,2}$),
involving localizations to  the sets
%$\cU_0$,
$\cV^{n,l}_0$, $\cW^{m,\ell}_0$. Let
$$m(\xi,s)= \widehat \psi(\xi) s^{-1} (e^{i s\xi}-1).$$
All estimates in proposition \ref{L2est} follow via  Plancherel's
theorem from the following set \eqref{k=0-L2est} of inequalities.
First, with $c_{n,l}$  as in \eqref{VlL2},
\begin{subequations}\label{k=0-L2est}
%\Be\label{U0L2}\sup_\xi \Big(\iint_{\cU_0} \frac{| m(\xi, s)-m(\xi,t)|^2 }{|s-t|^{2\alpha}} ds\,dt\Big)^{1/2} \lc 1
%\\ \label{V0L2}
%&\sup_\xi \Big(\iint_{\cV_0^{n,0}} \frac{| m(\xi, s)-m(\xi,t)|^2 }{|s-t|^{2\alpha}} ds\,dt\Big)^{1/2} \lc 2^{-n/2}.\end{align}
%which holds for $\alpha<3/2$.
%\Ee and, for $l\ge 0$,
\Be \label{V0L2} \sup_\xi \Big(\iint_{\cV_0^{n,l}} \frac{| m(\xi,
s)-m(\xi,t)|^2 }{|s-t|^{2\alpha}} ds\,dt\Big)^{1/2}\, \lc c_{n,l}\,.
\Ee
%\begin{cases}
%2^{\frac n2} 2^{l(\frac 32-\alpha)} &\text{ if } n\le 0, \, l\le n+2\\
%2^{-n/2} 2^{l(\frac 32-\alpha)}&\text{ if } n\ge 0, \, l\le 2\\
%2^{-n/2} 2^{-l(\alpha-\frac 12)}&\text{ if } n\ge 0, \, 0\le l\le n+2\\
%0 &\text{ if } l>n+2.
%\end{cases}\end{multline}
Next,
\begin{align} \label{W0-1L2}
&\sup_\xi \Big(\iint_{\cW_0^{n,\ell}} \frac{| m(\xi, s)|^2
}{|s-t|^{2\alpha}} ds\,dt\Big)^{1/2} \lc 2^{\ell/2} 2^{-n \alpha}
\min\{2^{-n/2}, 2^{n/2}\}\,,
\\
\label{W0-2L2} &\sup_\xi \Big(\iint_{\cW_0^{n,\ell}}
\frac{|m(\xi,t)|^2 }{|s-t|^{2\alpha}} ds\,dt\Big)^{1/2} \lc 2^{-n
(\alpha-\frac 12)} \min\{2^{-\ell/2}, 2^{\ell/2}\},
\end{align}
and (for $\ell+2\le n\le 0$) \Be\label{WL2}\sup_\xi
\Big(\iint_{\cW_0^{n,\ell}} \frac{| m(\xi, s)-m(\xi,t)|^2
}{|s-t|^{2\alpha}} ds\,dt\Big)^{1/2} \lc 2^{n(\frac 32-\alpha)}
2^{\frac\ell 2}\,. \Ee
\end{subequations}
We want to  {deduce \eqref{k=0-L2est}} from Lemma
\ref{sigmatauintegrallemma}. In view of the crucial  cancellation
property of $\psi$ we have \Be\label{psihatest} |\widehat
\psi(\xi)|\le C_N \frac{|\xi|^2} {(1+|\xi|)^N}\Ee for all $N$. Now
by a change of  {variables}
\[
\begin{aligned}& \Big(\iint_{\Omega} \frac{| m(\xi, s)-m(\xi,t)|^2 }{|s-t|^{2\alpha}}
ds\,dt \Big)^{1/2}
\\ &=|\xi|^\alpha|\widehat \psi(\xi)| \Big(\iint_{|\xi|\Omega}\Big |\frac{ e^{i\si}-1}{\si}
-\frac{ e^{i\tau}-1}{\tau} \Big|^2
\frac{d\si\,d\tau}{|\si-\tau|^{2\alpha}} \Big)^{1/2} .
\end{aligned} \]
%\begin{align*}& \Big(\iint_{\cU_0} \frac{| m(\xi, s)-m(\xi,t)|^2 }{|s-t|^{2\alpha}} ds\,dt\Big)^{1/2} \\ &=|\xi|^\alpha\Big(\iint_{|\xi|\cU_0}\Big |\frac{ e^{i\si}-1}{\si}    -\frac{ e^{i\tau}-1}{\tau} \Big|^2 \frac{d\si\,d\tau}{|\si-\tau|^{2\alpha}} \Big)^{1/2} \end{align*} and using \eqref{elemest1}, \eqref{elemest2} the last quantity is easily seen to be
%$O(|\xi|^{\alpha} (1+|\xi|^2))$which implies \eqref{U0L2}, by
%the required estimate follows from Lemma \ref{sigmatauintegrallemma}    \eqref{psihatest}.
Hence, by Lemma \ref{sigmatauintegrallemma} and \eqref{psihatest},
\begin{align*}& \Big(\iint_{\cV_0^{n,l}} \frac{| m(\xi, s)-m(\xi,t)|^2 }{|s-t|^{2\alpha}}
ds\,dt \Big)^{1/2}
\\ &\lc\frac{|\xi|^{\alpha+2}}{(1+|\xi|)^N}
\cdot \begin{cases} (2^{n}|\xi|)^{1/2} (2^l|\xi|)^{3/2-\alpha}
&\text{ if } 2^n|\xi|\le 1, \, l\le n+2,
\\
(2^n|\xi|)^{-1/2} (2^l|\xi|)^{3/2-\alpha} &\text{ if } 2^n|\xi|\ge
1, \, 2^l|\xi|\le 1,
\\
(2^n|\xi|)^{-1/2} (2^l|\xi|)^{-\alpha+1/2} &\text{ if } 2^n|\xi|\ge
1, \, 0\le l\le n+2,
\end{cases}
\end{align*} which implies \eqref{V0L2}.
The estimates \eqref{W0-1L2}, \eqref{W0-2L2}, \eqref{WL2} follow in
a similar way from Lemma \ref{sigmatauintegrallemma} and
\eqref{psihatest}.
\end{proof}
We finally note for further reference that summing the various
estimates in Proposition \ref{L2est} together with an application of
the Littlewood-Paley inequality (in $L^2$) yields the bound $\|\cT
F\|_2 \lc \|F\|_{L^2(\cH)}$.

\section{The $H^1 \to L^{1,\infty}$ bound}\label{hardy}
We shall follow the method outlined in \cite{sw} which has its root
in work by M.  Christ \cite{christ}.
%We need a few preliminaries on atomic decompositions.
We use a variant of the atomic decomposition  which also takes our
operator $T$ into account (by using the decomposition \eqref{Tdec}
and incorporating the Riesz potential operator in the atoms). The
approach here is based on the square-function characterization by
Chang and Fefferman \cite{crf} (in the one-parameter dilation
setting). {See} also \cite{se-stud} for an early  application to
endpoint estimates, and \cite{sw} for many  more references.

\subsection{\it Preliminaries} \label{preliminaries} Let $P_k$, $T_k$, $L_k$ as  in \eqref{Pk}, \eqref{Tk}, \eqref{Lk}{. We plan to} use the decomposition
\eqref{Tdec}. We consider the nontangential version of
 the Peetre maximal operators
  \begin{equation}\label{eq28}
\fM_k f(x)= \sup_{|h|\le 2^{-k}}|L_k f(x+h)|
\end{equation}
and the square function defined by \Be \label{fS} \fS f(x)=
\Big(\sum_{k\in \bbZ} |\fM_k f(x)|^2 \Big)^{1/2}.\Ee Then  (Peetre
\cite{peetre}) \Be\label{eq29} \|\fS f\|_{L^1} \lc \|f\|_{H^1} .\Ee
Let $\fJ_k$ be the set of dyadic intervals of length  $2^{-k}$ (i.e.
each interval is of the form $[n2^{-k}, (n+1)2^{-k})$ for some $n\in
\bbZ$). For $\mu\in \bbZ$ let  $$\cO_\mu= \{x: |\fS
f(x)|>2^{\mu}\}$$ and let $\fJ_k^\mu$ be the set of dyadic intervals
of length $2^{-k}$ with the property that $$|J\cap\cO_\mu|\ge |J|/2
\text{ and } |J \cap \cO_{\mu+1}|< |J| /2.$$ Clearly if $\fS f\in
L^1$ then every dyadic interval belongs to exactly one of the sets
$\fJ^\mu_k$.
%Let $$ e_J(f)=\bbone_J L_k f
%.$$
We then have (\cite{crf}) \Be\label{basicL2} \sum_{k\in \bbZ}
\sum_{J\in \fJ^\mu_k} \|\bbone_J L_k f\|_2^2 \lc 2^{2\mu} \meas
(\cO_\mu) \Ee For completeness we include the argument for
\eqref{basicL2}.  {The relevant fact is} that $|L_kf(x)|\le \fM_k
f(z)$ for all $x,z\in J$, for each $J\in \fJ_k$.
%for $J\in \fJ_k$  {and} each $x, z\in J$
%{one has}
Let
$$\cO_\mu^*= \{x: M_{HL} \bbone_{\cO_\mu} > 10^{-1}\}$$  where $M_{HL}$ stands for the Hardy-Littlewood maximal operator. Then  $$\meas(\cO_\mu^*)\lc \meas (\cO_\mu)$$ and we have $\cup_k  \cup_{J\in \fJ^\mu_k} J\subset \cO^*_\mu$.
Now
\begin{align*}
&\sum_{k\in \bbZ} \sum_{J\in \fJ^\mu_k} \|\bbone_J L_k f\|_2^2 \le
\sum_{k\in \bbZ} \sum_{J\in \fJ^\mu_k}2  \int_{J\setminus
\cO^{\mu+1}}|\fM_k f(x)|^2 dx
\\
&\le 2 \int_{\cO^*_\mu\setminus \cO^{\mu+1}} \sum_{k\in \bbZ} |\fM_k
f(x)|^2 dx \le 2^{2\mu+2} \meas(\cO_\mu^*) \le C 2^{2\mu}
\meas(\cO_\mu)
\end{align*} which establishes  \eqref{basicL2}.

Now we assign to each dyadic interval $J$ another dyadic interval
$I(J)$ containing $J$.  {If  $J\in \fJ^\mu_k$ then clearly $J\subset
\cO^*_\mu$. Let} $I(J)$ be the maximal dyadic interval containing
$J$ which is contained in $\cO_\mu^*$.
 {Set}
$$b^{\mu,I}_{k}(x) = \sum_{\substack {J\in \fJ_k^\mu:\\I(J)=I}}
L_kf(x)\bbone_J(x).
$$
 We write
 $L(I)=L$ if the length of a  dyadic interval $I$ is $2^{L}$.
  Also we let $\cI^\mu$ be the collection of all dyadic intervwhich are maximal and contained in $\cO^*_\mu$. By the maximality condition  the intervals in $\cI^\mu$ have disjoint interior. For future reference we note that if $J\in \fJ^\mu_k$ and $I(J)=I$ then $L(I)+k\ge 0$.

 {Set}, for $I\in \cI^\mu$,
\Be \label{gammadef}\gamma_{\mu,I}:=\Big( \sum_{\substack
{k:\\k+L(I)\ge 0}}\sum_{J\in \fJ^k_\mu} \| \bbone_J L_k
f\|_2^2\Big)^{1/2}.\Ee {We have}
\begin{align*}
&\sum_{I\in \cI^\mu} |I|^{1/2} \gamma_{\mu,I} \le  \Big(\sum_{I\in
\cI^\mu}|I|\Big)^{1/2}\Big(  \sum_{I\in \cI^\mu}
\gamma_{\mu,I}^2\Big)^{1/2}
%\notag
\\
&\lc |\cO_\mu^*|^{1/2} (2^{2\mu} |\cO_\mu|)^{1/2} \lc 2^\mu
|\cO_\mu|
%\notag
\end{align*}
and hence \Be \label{atomicest} \sum_{\mu\in \bbZ} \sum_{I\in
\cI^\mu} |I|^{1/2} \gamma_{\mu,I} \lc \sum_{\mu\in \bbZ} 2^\mu
|\cO_\mu| \lc \|\fS f\|_1\lc \|f\|_{H^1}, \Ee {which} is equivalent
to
%n particular this implies
\Be\label{b-atomicest} \sum_{\mu\in \bbZ} \sum_{I\in
\cI^\mu}|I|^{1/2}  \Big(\sum_{k\in \bbZ}
\|b_{k}^{\mu,I}\|_2^2\Big)^{1/2} \lc \|f\|_{H^1}. \Ee

%{\it Remark}
%The functions $\big\{a_{\mu,I}\big\}_{{k\in \bbZ} :=|I|^{-1/2} \gamma_{\mu,I}I^{-1} b_I$ are usually referred to as ``$2$-atoms". $a_I$ is supported on $I$, has vanishing moments up to order $100$ and satisfies $\|a_I\|_2\le |I|^{-1/2}$. Setting $c_I= |I|^{1/2} \ga_I$ we have $f=\sum_I c_I a_I$ with $\sum_I |c_I| \lc \|f\|_{H^1}$.

\medskip

\subsection{\it Proof of the $H^1\to L^{1,\infty}$  inequality} Fix $\la>0$.
We  {claim} that \Be\label{wtgoal} \meas \big(\{x\in \bbR:
\,|Tf(x,\cdot, \cdot)|_{L^2(\bbR^2)}
>10\la \}\big) \lc \la^{-1} \sum_{\mu\in \bbZ} \sum_{I\in \cI^\mu} |I|^{1/2} \gamma_{\mu,I}
\Ee which implies the desired bound, by \eqref{atomicest}.

The first step is the definition of an exceptional set $\cE$.  Given
any $I,\mu$  {with $\mu\in \bbZ$, $I\in \fJ^{\mu}$, we} assign an
integer  $\ka(\mu,I)$ (depending on $\la$), defined  as
$$\ka(\mu,I)=\max \{L(I), \widetilde \vk(\mu,I)\}\, $$ where  the
"stopping time" $\widetilde \vk(\mu,I)$ is given by
% $\widetilde \tau(I)$ is   the smallest integer $\tau$  for
\Be\label{tildeI}\widetilde  \vk(\mu,I)= \inf\{r\in \bbZ: \, 2^{r}
\ge \la^{-1}|I |^{1/2}\gamma_{\mu,I}\,\}.\Ee
%(or $-\infty$ if there is no such smallest integer).

For any $I,\mu$ satisfying $L(I)<\kappa(\mu, I)$  let $\cE_{\mu, I}$
be the interval of length $2^{\ka(\mu, I)+5}$,  {concentric with}
$I$ and let
$$\cE=\bigcup_{\substack{\mu, I:  I\in \cI^\mu\\L(I)<\kappa(\mu,I)} } \cE_{\mu,I}.$$  {For} any $I$ with $\ka(\mu, I)> L(I)$, we have
$2^{\kappa(\mu, I)-1}\le \la^{-1}|I|^{1/2}\ga_{\mu,I}$. Thus
% by the minimality condition in the definition of $\tau$. Thus
\Be\label{excsetest} \meas(\cE)\lc \sum_{I,\mu:L(I)< \ka(\mu, I)}
2^{\ka(\mu,I)} \, \lc \sum_{I,\mu}
 \la^{-1}  |I|^{1/2}\gamma_{\mu,I}  \lc  \la^{-1}\|f\|_{H^1}.\Ee

Hence in order to prove \eqref{wtgoal} we only need to show \Be
\label{Tmeasc}\meas(x\in \cE^\complement:  \,|Tf(x,\cdot,
\cdot)|_{L^2(\bbR^2)} >10\la\}
 \lc \la^{-1}  \sum_{\mu,I} |I|^{1/2}\gamma_{\mu, I}.
 \Ee
 {By Minkowski's inequality  we have}
\Be
\begin{aligned}
&|Tf(x,\cdot, \cdot)|_{L^2(\bbR^2)} = \Big(\iint\Big |\sum_{\mu, I}
\sum_{-k\le L(I)}  P_k T_k b_{k}^{\mu,I} (x,s,t)\Big|^2 ds\,
dt\Big)^{1/2}
\\&\le \sum_{i=1}^4U_i(x) + \cV_{\cE}(x)+\cW_{2, \cE}(x) + V_*(x)+ W_{2,*}(x) +W_1(x)
\end{aligned}
\Ee where
\begin{subequations}
\begin{align}
U_1(x)&\!=\!\Big(\iint\Big| \sum_{\substack{n \\\ell\le n+2}}
\sum_{k:(s,t)\in \cV_k^{n,l} } \sum_{\substack{\mu, I:\\ I\in
\cI^\mu\\ -k+n\le L(I)}} P_k T_k b_{k}^{\mu,I} (x,s,t)\Big|^2ds\,
dt\Big)^{1/2},
\\
U_2(x)&\!=\!\Big(\iint\Big| \sum_{\substack{n\le 2\\\ell\le n-2}}
\sum_{k:(s,t)\in \cW_k^{n,\ell}}
 \sum_{\substack{\mu, I: I\in \cI^\mu\\ -k\le L(I)}}
P_k T_k b_{k}^{\mu,I} (x,s,t)\Big|^2ds\, dt\Big)^{1/2},
\end{align}
\begin{align}
U_{3}(x)&\!=\!\Big(\iint\Big| \sum_{\substack{n\ge 2\\\ell\le n-2}}
\sum_{k:(s,t)\in \cW_k^{n,\ell}}
 \sum_{\substack{\mu, I: I\in \cI^\mu\\ -k+n \le L(I)}}
P_k T_{k,1} b_{k}^{\mu,I} (x,s,t)\Big|^2ds\, dt\Big)^{1/2},
\\U_{4}(x)&\!=\!\Big(\iint\Big|
\sum_{\substack{n\ge 2\\\ell\le n-2}} \sum_{k:(s,t)\in
\cW_k^{n,\ell}}
 \sum_{\substack{\mu, I: I\in \cI^\mu\\ -k+\ell \le L(I)}}
P_k T_{k,2} b_{k}^{\mu,I} (x,s,t)\Big|^2ds\, dt\Big)^{1/2}
\end{align}
\end{subequations} and
\begin{subequations} \begin{align}
V_\cE(x)&= \Big(\iint\Big| \sum_{\substack{n,l :\\ n\ge \max\{0,
l-2\} }}\sum_{k:(s,t)\in \cV_k^{n,l} } \sum_{\substack{\mu, I:
\\I\in \cI^\mu\\ L(I)< -k+n\le \ka(\mu,I)}} P_k T_k b_{k}^{\mu,I}
(x,s,t)\Big|^2ds\, dt\Big)^{1/2},
%{\color{red}and}
\\
W_{2,\cE}(x)&= \Big(\iint\Big| \sum_{\substack{n,\ell :\\ n\ge
\max\{0, \ell+2\}}} \sum_{k:(s,t)\in \cW_k^{n,\ell} }
\sum_{\substack{\mu, I:\\ I\in \cI^\mu\\ L(I)< -k+\ell\le
\ka(\mu,I)}} P_k T_{k,2} b_{k}^{\mu,I} (x,s,t)\Big|^2ds\,
dt\Big)^{1/2}.
\end{align}
\end{subequations}
Furthermore
\begin{subequations}
\begin{align}
V_*(x)&= \Big(\iint\Big| \sum_{\substack{n,l :\\ n\ge \max\{0,
l-2\}} }\sum_{k:(s,t)\in \cV_k^{n,l} } \sum_{\substack{\mu, I:
\\I\in \cI^\mu\\  \ka(\mu,I)<-k+n}} P_k T_k b_{k}^{\mu,I}
(x,s,t)\Big|^2ds\, dt\Big)^{1/2},
\\
W_{2,*}(x)&= \Big(\iint\Big| \sum_{\substack{n,\ell :\\ n\ge
\max\{0, \ell+2\}}} \sum_{k:(s,t)\in \cW_k^{n,\ell} }
\sum_{\substack{\mu, I:\\ I\in \cI^\mu\\  \ka(\mu,I)< -k+\ell}} P_k
T_{k,2} b_{k}^{\mu,I} (x,s,t)\Big|^2ds\, dt\Big)^{1/2}
\end{align}
\end{subequations}
and finally \Be W_{1}(x)\!=\! \Big(\iint\Big|\!
\sum_{\substack{n,\ell :\\ n\ge \max\{0, \ell+2\}}} \sum_{k:(s,t)\in
\cW_k^{n,\ell} } \! \sum_{\substack{\mu, I:\\ I\in \cI^\mu\\  L(I)<
-k+n}}\! P_k T_{k,1} b_{k}^{\mu,I} (x,s,t)\Big|^2ds\, dt\Big)^{1/2}.
\Ee {The} quantity on the left hand side of \eqref{Tmeasc} is {not
greater than}
 \begin{align*}
& \sum_{i=1}^4 \meas(x\in \cE^\complement:  U_i(x) >\la\}+
\meas(x\in \cE^\complement:  W_1(x) >\la\}
\\
& +\meas(x\in \cE^\complement:  V_*(x) >\la\}+
 \meas(x\in \cE^\complement:  W_{2,*}(x) >\la\}
\\& + \meas(x\in \cE^\complement:  V_\cE(x) >\la\}+
 \meas(x\in \cE^\complement:  W_{2,\cE}(x) >\la\}.
  \end{align*}
{The} terms $V_\cE$ and $W_{2,\cE}$ are supported in $\cE$ and are
thus irrelevant for the estimate \eqref{Tmeasc}. Thus \eqref{Tmeasc}
follows, by Tshebyshev's inequality, from the bounds
\begin{align}
  \label{UiL1}\sum_{i=1}^4 \|U_i\|_1&\lc \sum_{\mu,I}|I|^{1/2}\gamma_{\mu,I},
\\
\label{V*L2}\|V_*\|_2^2 &\lc \la
\sum_{\mu,I}|I|^{1/2}\gamma_{\mu,I},
\\
\label{W*L2}\|W_{2,*}\|_2^2 &\lc \la
\sum_{\mu,I}|I|^{1/2}\gamma_{\mu,I},
\\
\label{W1L1}\|W_1\|_1 & \lc  \sum_{\mu,I}|I|^{1/2}\gamma_{\mu,I}.
\end{align}

\subsection*{\it Proof of \eqref{UiL1}}
 {For} $(s,t)\in \cV_k^{n,l} $ and $-k+n\le L(I)$ the function $x\mapsto
P_k T_k b_{k}^{\mu,I} (x,s,t)$ is supported in a tenfold expansion
$I^*$ of $I$. We use Minkowski's inequality for the $n,l,\mu, I$
sums, and  then Cauchy-Schwarz on $I^*$ to get
%and then
\begin{align*}
\|U_1\|_1 &\le \sum_{\substack{n \\ l \le n+2} }
\sum_{\substack{\mu, I:\\ I\in \cI^\mu}}\Big\| \Big(\iint\Big|
\sum_{\substack{k:(s,t)\in \cV_k^{n,l} \\ -k+n\le L(I)}} P_k T_k
b_{k}^{\mu,I} (\cdot,s,t) \Big|^2ds\, dt\Big)^{1/2}\Big\|_1
\\&{\lesssim} \sum_{\substack{n \\ l \le n+2} }
\sum_{\substack{\mu, I:\\ I\in \cI^\mu} }|I|^{1/2} \Big(\iint\Big\|
\sum_{\substack{k:(s,t)\in \cV_k^{n,l} \\ -k+n\le L(I)}} P_k T_k
b_{k}^{\mu,I} (\cdot,s,t)\Big\|_2^2ds\, dt\Big)^{1/2}.
\end{align*}
%%%%%%%%%%%%%%%%%%%%%%%%%%%%%%%%%%
 Denote by $c_{n,l}$ the constants defined in
%appearing in
%the right hand side of
 \eqref{VlL2}. {Then}
$\sum_{n} \sum_{l\le n+2} c_{n,l} < \infty$. Now we apply
\eqref{VlL2} to get
\begin{align*}
\|U_1\|_1  & \le  \sum_{\substack{n \\ l \le n+2} }  c_{n,l}
\sum_{\substack{\mu, I:\\ I\in \cI^\mu} }|I|^{1/2}
\Big(\sum_{k:-k\le L(I)}
\big\|b_{k}^{\mu,I} \big\|_2^2\Big)^{1/2} \\
& \lc \sum_{\substack{\mu, I:\\ I\in \cI^\mu} }|I|^{1/2}
\Big(\sum_{k:-k\le L(I)} \big\|b_{k}^{\mu,I} \big\|_2^2\Big)^{1/2}.
\end{align*}
%From this it follows, by  the almost orthogonality of the $P_k$,
%\begin{align*} \|U_1\|_1
%&\lc \sum_{n,l\le n-2}
%\sum_{\substack{\mu, I:\\ I\in \cI^\mu} }|I|^{1/2}
%\Big(\sum_{k:-k+n\le L(I)} \iint_{ \cV_k^{n,l} }
%\big\|T_k b_{k}^{\mu,I} (\cdot,s,t)\big\|_2^2ds\, dt\Big)^{1/2}
%\\
%&\lc \sum_{n,l}
%\sum_{\substack{\mu, I:\\ I\in \cI^\mu} }|I|^{1/2}
%\Big(\sum_{k:-k\le L(I)}
%\big\|b_{k}^{\mu,I} \big\|_2^2\Big)^{1/2}
%\end{align*}
%where for the last inequality we have used \eqref{VlL2}.
%%%%%%%%%%%%%%%%%%%%%%%%%%%%%%%%%%%%%%%%%%%
We apply {a} similar {argument} to estimate the  $L^1$ norms
of~$U_2$, $U_3$, $U_4$. For~$U_2$ we get
\begin{align*} \|U_2\|_1&\le \sum_{\substack{n\le 2\\\ell\le n-2}}
\sum_{\substack{\mu, I:\\ I\in \cI^\mu} }|I|^{1/2} \Big( \iint\Big\|
\sum_{\substack{k:(s,t)\in \cW_k^{n,l} \\ -k\le L(I)}} P_k T_k
b_{k}^{\mu,I} (\cdot,s,t)\Big\|_2^2ds\, dt \Big)^{1/2}
\\
&\lc\sum_{\substack{\mu, I:\\ I\in \cI^\mu} }|I|^{1/2}
\Big(\sum_{k:-k\le L(I)} \big\| b_{k}^{\mu,I} \big\|_2^2\Big)^{1/2}
\end{align*}
where we used \eqref{WellL2}. {By \eqref{WellL2-1}}
\begin{align*}
\|U_3\|_1 &\le \sum_{\substack{\mu, I:\\ I\in \cI^\mu}
}|I|^{1/2}\sum_{\substack{n\ge 2\\\ell\le n-2}} \Big(\iint\Big\|
\sum_{\substack{k:(s,t)\in \cW_k^{n,l} \\ -k+n\le L(I)}} P_k T_{k,1}
b_{k}^{\mu,I} (\cdot,s,t)\Big\|_2^2ds\, dt\Big)^{1/2}
\\
&\lc\sum_{\substack{\mu, I:\\ I\in \cI^\mu} }|I|^{1/2}
\Big(\sum_{k:-k\le L(I)} \big\| b_{k}^{\mu,I} \big\|_2^2\Big)^{1/2},
\end{align*}
{and, by \eqref{WellL2-2},}
\begin{align*}
\|U_4\|_1 &\le \sum_{\substack{\mu, I:\\ I\in \cI^\mu}
}|I|^{1/2}\sum_{\substack{n\ge 2\\\ell\le n-2}} \Big( \iint\Big\|
\sum_{\substack{k:(s,t)\in \cW_k^{n,l} \\ -k+l\le L(I)}} P_k T_{k,2}
b_{k}^{\mu,I} (\cdot,s,t)\Big\|_2^2ds\, dt \Big)^{1/2}
\\
&\lc\sum_{\substack{\mu, I:\\ I\in \cI^\mu} }|I|^{1/2}
\Big(\sum_{k:-k\le L(I)} \big\| b_{k}^{\mu,I} \big\|_2^2\Big)^{1/2}.
\end{align*}
{Finally} we use $\sum_{k} \big\| b_{k}^{\mu,I} \big\|_2^2 =
(\gamma_{\mu}^I)^2$  in all estimates above to complete the proof of
\eqref{UiL1}.

%\begin{align*}\sum_{\substack{\mu, I:\\ I\in \cI^\mu} }|I|^{1/2}\Big(\sum_{k:-k\le L(I)} \big\| b_{k}^{\mu,I} \big\|_2^2\Big)^{1/2}&\lc\sum_\mu|\cO_\mu^*|^{1/2} \Big(\sum_{\substack{ I:\\ I\in \cI^\mu} }\sum_{k:-k\le L(I)} \sum_{\substack{J\in \fJ^\mu_k\\I(J)=I}}\big\| L_k f \bbone_J \big\|_2^2\Big)^{1/2}\\\\&\lc\sum_\mu|\cO_\mu| 2^\mu \,\lc\, \|f\|_{H^1}.\end{align*}

\subsection*{\it Proof of \eqref{V*L2}}
We have,  by Minkowski integral inequality,
\begin{align*}
V_*(x) &\le\!  \sum_{\substack{n,l \\\ n\ge \max\{0, l-2 \}}} \!
\!\Big(\iint_{\cV_k^{n,l}}\Big| \!\sum_{k:(s,t)\in \cV^{n,l}_k}\!
\sum_{\substack{\mu, I: \\I\in \cI^\mu\\  \ka(\mu,I)<-k+n}}\!\! P_k
T_k b_{k}^{\mu,I} (x,s,t)\Big|^2ds\, dt\Big)^{1/2}
\\
&\le\!  \sum_{\substack{n,l \\\ n\ge \max\{0, l-2 \}}}  \Big(\sum_k
\iint_{\cV_k^{n,l}}\Big| \sum_{\substack{\mu, I: \\I\in \cI^\mu\\
\ka(\mu,I)<-k+n}} P_k T_k b_{k}^{\mu,I} (x,s,t)\Big|^2ds\,
dt\Big)^{1/2}
\end{align*}
and so, by Fubini,
\begin{equation*}
\|V_*\|_2 \lc \sum_{\substack{n,l :\\ n\ge \max\{0, l-2\}}}
 \Big(\sum_k\iint_{\cV^{n,l}_k} \Big\|
\sum_{\substack{\mu, I: I\in \cI^\mu\\  \ka(\mu,I)<-k+n}} P_k T_k
b_{k}^{\mu,I} (\cdot,s,t) \Big\|_2^2 \,ds\, dt\Big)^{1/2} .
\end{equation*}
By \eqref{VlL2},
\begin{multline*}
\|V_*\|_2\lc \sum_{l} \min\{2^{-l(\alpha-1/2)}, 2^{l(3/2-\alpha)}\}
\sum_{n\ge \max\{0,l-3\}} 2^{-n/2}\\
{\times} \Big(\sum_k \Big\| \sum_{\substack{\mu, I: I\in \cI^\mu\\
\ka(\mu,I)<-k+n}} b_{k}^{\mu,I} \Big\|_2^2\Big)^{1/2}.
\end{multline*}
For fixed $k$,
\begin{align*}
&\Big\|\sum_\mu \sum_{\substack{I\in \cI^\mu:\\ \ka(\mu, I)< -k+n} }
b_k^{\mu,I}\Big\|_2^2 = \Big\| \sum_\mu\sum_{\substack{\mu,I:\\
\ka(\mu, I)< -k+n}}  \sum_{\substack {J\in \fJ^\mu_k\\ I(J)=I}}
\bbone_J L_k f \Big\|_2^2
\\&=\sum_\mu\sum_{\substack{I\in \cI^\mu\\ \ka(\mu, I)< -k+n} }
 \sum_{\substack {J\in \fJ^\mu_k\\ I(J)=I}}
\big \|\bbone_J L_k f\big \|_2^2
\end{align*}
{because each dyadic interval of length~$2^{-k}$ is contained in
exactly one family~$\fJ^\mu_k$, and for fixed $\mu$ the intervals in
$\cI^\mu$ have disjoint interior.} Now, since $1/2<\alpha<3/2$ , we
can sum in $l$ and obtain
\begin{align*}
\|V_*\|_2^2 &\lc \Big( \sum_{n\ge 0} 2^{-n}\Big(\sum_k
\sum_{\substack{\mu, I: I\in \cI^\mu\\  \ka(\mu,I)<-k+n}}
\sum_{\substack {J\in \fJ^\mu_k\\ I(J)=I}}
 \big\|\bbone_J L_k f\big\|_2^2\Big)^{1/2}\Big)
\\
&\lc \sum_k \sum_{\substack{\mu, I:\\ I\in \cI^\mu}}
2^{-(\ka(\mu,I)+k)} \sum_{\substack {J\in \fJ^\mu_k\\ I(J)=I}}
 \big\|\bbone_J L_k f\big\|_2^2
\\
&\lc \sum_{\substack{\mu, I:\\ I\in \cI^\mu}} 2^{-(\ka(\mu,I)-L(I))}
\sum_{\substack {J\in \fJ^\mu_k\\ I(J)=I}} 2^{-(L(I)+k)}
 \big\|\bbone_JL_k f\big\|_2^2
\end{align*}
and hence \Be \label{basicV*est}
\begin{aligned}
\|V_*\|_2^2&\lc\sum_{\mu, I:I\in \cI^\mu} 2^{L(I)-\kappa(\mu, I)}
\gamma_{\mu,I}^2
\\&\lc\sum_\mu \Big[
\sum_{\substack{ I\in \cI^\mu\\
\ka(\mu,I)=L(I)}}
 \gamma_{\mu,I}^2+
\sum_{\substack{ I\in \cI^\mu\\
\ka(\mu,I)>L(I)}} 2^{L(I)-\ka(\mu,I)} \gamma_{\mu,I}^2\Big].
\end{aligned}
\Ee If $\ka(\mu,I)=L(I)$ then $\widetilde \vk(\mu,I)\le L(I)$ and by
definition of $\widetilde \vk(\mu, I)$ we {then have}
 $2^{L(I)}\ge 2^{\widetilde \vk(\mu,I)} \ge \la^{-1}|I|^{1/2}\gamma_{\mu,I}${. Thus}
$\gamma_{\mu,I}\le |I|^{1/2} \la$. Therefore
\begin{subequations}\label{Lversuska}
\Be\label{Lversuska=} \sum_{\substack{I\in \cI^\mu:\\
\ka(\mu,I)=L(I)} } \gamma_{\mu,I}^2 \le \la \sum_{I\in \cI^\mu}
|I|^{1/2} \gamma_{\mu,I}.\Ee If $\ka(\mu,I)> L(I)$ then
$\ka(\mu,I)=\widetilde \vk(\mu,I)$ and thus $2^{-\ka(\mu,I)} \le \la
|I|^{-1/2}\gamma_{\mu,I}^{-1}$, again by the definition of
$\widetilde \vk(\mu,I)$.
 Hence
\Be \label{Lversuska<}\sum_{\substack {I\in \cI^\mu\\
\ka(\mu,I)>L(I)} }2^{L(I)-\ka(\mu,I)} \gamma_{\mu,I}^2 \lc \la
\sum_{I\in \cI^\mu}  |I|^{1/2}\gamma_{\mu, I}.\Ee
\end{subequations}
Now combining \eqref{basicV*est} with \eqref{Lversuska=},
\eqref{Lversuska<} completes the proof of \eqref{V*L2}.

\subsection*{\it Proof of \eqref{W*L2}} This proof follows the lines of that of \eqref{V*L2}.
Notice that the conditions $\ell\ge \ka(\mu,I)+k \ge  L(I)+k$  imply
that $\ell\ge 0$.
 Now \begin{align*}
W_{2,*}(x) & \le  \\
&   \sum_{\substack{n,\ell \\ n\ge \max\{0, \ell+2\}}} \Big(\sum_k
\iint_{\cW_k^{n,l}}\Big| \sum_{\substack{\mu, I: I\in \cI^\mu\\
\ka(\mu,I)< -k+\ell}} P_k T_{k,2} b_{k}^{\mu,I} (x,s,t)\Big|^2ds\,
dt \Big)^{1/2}
\end{align*}
and so
\begin{align*}
\|W_{2,*}\|_2 & \le \\
&  \sum_{\substack{n,\ell \\ n\ge \max\{0, \ell+2\}}}
\Big(\sum_{k}\iint_{\cW_k^{n,l}}\Big\| \sum_{\substack{\mu, I: I\in
\cI^\mu\\  \ka(\mu,I)< -k+\ell}} P_k T_{k,2} b_{k}^{\mu,I}
(\cdot,s,t)\Big\|_2^2ds\, dt \Big)^{1/2}\,.
  \end{align*}
 By \eqref{WellL2-2}

\begin{align*}
\|W_{2,*}\|_2^2  & \lc
 \sum_{n\ge 0} 2^{-n(2\alpha-1)} \sum_{\ell\ge 0} 2^{-\ell}
 \sum_k \Big\|
\sum_{\substack{\mu, I:I\in \cI^\mu\\  \ka(\mu,I)< -k+\ell}}
 b_{k}^{\mu,I} \Big\|_2^2
  \\&\lc
\sum_{\ell\ge 0} 2^{-\ell}  \sum_k
 \sum_{\substack{\mu, I:I\in \cI^\mu\\  \ka(\mu,I)< -k+\ell}}
 \sum_{\substack{J\in \fJ^\mu_k\\ I(J)=I}} \big\|\bbone_J L_k
 f\big\|_2^2\,,
 %\ga_{\mu,I}^2
 \end{align*}
{which} implies
 \begin{align*}
 \|W_{2,*}\|_2^2
 &\lc \sum_k\sum_{\substack{\mu, I:\\ I\in \cI^\mu}} 2^{-k-\ka(\mu,I)}\sum_{\substack{J\in \fJ_\mu^k\\I(J)=I}}\|\bbone_J L_kf\|_2^2
 \lc \sum_{\substack {\mu, I:\\ I\in \cI^\mu} } 2^{L(I)-\ka(\mu,I)} \gamma^2_{\mu,I}
 \end{align*}
 and this expression has been already estimated by $C\la \sum_{\mu, I} |I|^{1/2} \gamma_{\mu,I}$,
{by~\eqref{basicV*est}, \eqref{Lversuska=} and \eqref{Lversuska<}}.
This finishes the proof of
 \eqref{W*L2}.

\subsection*{\it Proof of \eqref{W1L1}}
We now take advantage of the fact that the $L^2$ bounds for
$T_{k,1}$ in \eqref{WellL2-1} are somewhat better then the
corresponding bounds for $T_{k,2}$ in \eqref{WellL2-2}. This allows
us to invoque a straightforward $L^1$ estimate for $W_1$ as opposed
to the $L^2$ arguments used for $V_*$ and $W_{2,*}$. We have

\begin{align*}
&\|W_1\|_1\\& =  \Big\| \Big(\iint\Big| \sum_{\substack{n,\ell :\\
n\ge \max\{0, \ell+2\}}}
 \sum_{k:(s,t)\in \cW_k^{n,\ell} }
\sum_{\substack{\mu, I: I\in \cI^\mu\\  L(I)< -k+n}} P_k T_{k,1}
b_{k}^{\mu,I} (\cdot,s,t)\Big|^2ds\, dt\Big)^{1/2}\Big\|_1
\\
&\lc \sum_{\substack{n,\ell :\\ n\ge \max\{0, \ell+2\}}}
\sum_{\substack{\mu, I:\\ I\in \cI^\mu}} \sum_{k: L(I)-n<-k\le L(I)}
\Big\| \Big(\iint_{\cW_k^{n,\ell}}\Big| P_k T_{k,1} b_{k}^{\mu,I}
(\cdot,s,t)\Big|^2ds\, dt\Big)^{1/2}\Big\|_1.
\end{align*}
%}
Now observe that the expression inside $\|\cdots\|_1$ is supported
in an interval of length $2^{-k+n+5}$, {concentric with} $I$. Hence,
by the Cauchy-Schwarz {inequality and Fubini}
%{\color{red}
\begin{multline*}
\|W_1\|_1\lc\\
\sum_{\substack{n,\ell :\\ n\ge \max\{0, \ell+2\}}}
\sum_{\substack{\mu, I:\\ I\in \cI^\mu}}
\sum_{\substack{L(I)-n<\\-k\le L(I)}}2^{\frac{n-k}2}
\Big(\iint_{\cW_k^{n,\ell}}\big\| P_k T_{k,1} b_{k}^{\mu,I}
(\cdot,s,t)\big\|_2^2ds\, dt\Big)^{1/2}.
\end{multline*}
%}
By \eqref{WellL2-1},
%this is bounded by  a constant times
\begin{align*}\|W_1\|_1&\lc
\sum_{\substack{n,\ell :\\ n\ge \max\{0, \ell+2\}}} 2^{\ell/2}
2^{-n(\alpha+\frac 12)} \sum_{\substack{\mu, I:\\ I\in \cI^\mu}}
\sum_{\substack{L(I)-n<\\-k\le L(I)}}2^{\frac{n-k}2}
\|b_{k}^{\mu,I}\|_2
\\
&\lc \sum_{n\ge 0} 2^{-n(\alpha-1/2)}\sum_{\substack{\mu, I:\\ I\in
\cI^\mu}}\sum_{-k\le L(I)} 2^{-(k+L(I))/2} |I|^{1/2}
\|b_{k}^{\mu,I}\|_2
\end{align*}
and it follows easily that
$$\|W_1\|_1 \lc
\sum_\mu\sum_{I\in \cI^\mu}|I|^{1/2} \gamma_{\mu,I}\,,$$ as claimed.
\qed

\bigskip

\section{$L^p$ estimates}\label{Lpsect}
\subsection{\it Proof of Theorem \ref{sobchar}, via  estimates on Hardy spaces.}
The lower bounds have already been established in
\S\ref{lowerbounds}. For the upper bounds we need to {distinguish}
the {case} $1<p<2$ (for which the result is an immediate consequence
of what we have already proved) and the case $2<p<\infty$.

\subsubsection{The case $1<p<2$}
 For the upper $L^p$ bounds we note that
\Be\label{Lpbd}\|S_\alpha (\cD^{-\alpha} f) \|_{L^{p} } \le C_p
\|f \|_{L^p} {,\quad 1<p<2,} \Ee follows by real interpolation
(\cite{frs}) from the already proved bounds
\begin{align*}
&\|S_\alpha (\cD^{-\alpha} f) \|_{L^{1,\infty} } \lc \|f \|_{H^1},
\\
&\|S_\alpha (\cD^{-\alpha} f) \|_{L^{2}} \lc \|f \|_{L^2}.
\end{align*}

\subsubsection{The case $2<p<\infty$}
{Consider} the operator $\cT$ acting on the
$\cH=L^2(\bbR\times\bbR)$ valued functions $F$ by \Be
\label{Tdualdef} \cT F(x) =  \iint_{|t|\le |s|}
\frac{s^{-1}\Delta_{-s}\cD^{-\alpha} F(x,s,t)-t^{-1}\Delta_{-t}
\cD^{-\alpha} F(x,s,t) }{|s-t|^{\alpha}}  \, ds\, dt\,. \Ee and
observe that \eqref{Lpbd} for $2\le p<\infty$ follows {by duality}
from \Be\label{Lpbddual}\|\cT F \|_{L^{p}} \le C_p
\|F\|_{L^p(\cH)}, \quad 1<p\le 2.\Ee This can be deduced by real
interpolation  from
 \Be\label{L2dual}
 \|\cT F \|_{L^{2}} \lc  \|F\|_{L^2(\cH)}\Ee
 (which is equivalent to the case $p=2$ of \eqref{Lpbd})
 and
 \begin{theorem}\label{dualH1thm}
$$ %\meas\big(\{x: |\cT f(x)|>\la \}\big)
\|\cT F\|_{L^{1,\infty}}\lc \|F\|_{H^1(\cH)} .$$
\end{theorem}
This result follows from
\begin{multline} \label{dualhardy}
\meas (\{x: |\cT F(x)| >\la \})  \\ \le \frac{C}{\la} \,
\Big\|\Big(\sum_{k\in \bbZ} \sup_{|h|\le 2^{-k}}  \iint |L_k
F(\cdot+h,s,t) |^2 ds\, dt \Big)^{1/2} \Big\|_1
\end{multline}
where the $L^1$ norm on the right hand side involves a version of
the maximal square function $\fS F$ in \eqref{fS}, but for
$\cH$-valued functions $F$. {More precisely, in~\eqref{eq28} one
should replace the absolute value by the norm in~$\cH$. Then
Peetre's estimate~\eqref{eq29} holds in this context.} The proof of
\eqref{dualhardy} will be omitted since it is essentially the same
as the proof of Theorem~\ref{sobhardy}, with appropriate notational
modifications.
%We shall omit the details.

\subsection{\it An alternative  approach to Theorem \ref{sobchar}}\label{altLpproof}
 There is an  alternative (more straightforward and direct, but not less lengthy) approach  to
Theorem \ref{sobchar} which bypasses Theorem \ref{sobhardy}.
% is outlined in \S\ref{alt-Lpsect}.
%In  \S\ref{upperbounds} we prove the $L^p$ bounds for the square-function  $S_\alpha$.
%It involves checking bounds for the so called H\"ormander conditions for several constituents in the decompositions of the
%last section.

To be specific we let $\phi$  be a $C^\infty$ function supported in
$\{\xi:1/2<|\xi|\le 2\}$ and let $\Phi=\cF^{-1}[\phi]$. Let
$K_k(x,s,t )$ be defined by
$$
\widehat K_k(\xi,s,t)=2^{-k\alpha} |s-t|^{-\alpha}\phi(2^{-k}|\xi|)
\Big( \frac{e^{is\xi}-1}{s} - \frac{e^{it\xi}-1}{t}\Big)\,.
$$
By Littlewood-Paley theory one reduces the proof of Theorem
\ref{sobchar} to the following $L^p$  inequalities for $1<p<2$:
\[
\Big\| \Big(\sum_{k\in \bbZ}  \iint_{|t|\le |s|}\Big | \int
K_k(\cdot-y,s,t) f_k(y) dy\Big |^2ds dt\Big)^{1/2}\Big\|_p \le C_p
\Big\|\Big(\sum_k|f_k|^2\Big)^{1/2} \Big\|_p,
\]
and
\begin{multline*}\Big\| \Big(\sum_{k\in \bbZ} \Big|  \iint_{|t|\le |s|}  \int K_k(\cdot-y,s,t)F_k(y, s,t) dy\,ds\, dt\Big|^2 \Big)^{1/2}\Big\|_p
\\
\le C_p \Big\|\Big(\sum_k\iint |F_k(\cdot,s,t)|^2 ds\, dt\Big)^{1/2}
\Big\|_p.
\end{multline*}
One decomposes, for each $k$, the half plane   $\{|t|\le |s|\}$ as a
union of $\cV^{n,l}_k$ and $\cW^{n,\ell}_k$,  as in \S\ref{L2sect}.
One then aims to prove, for $1<p\le 2$, {that there is} $\eps(p)>0$,
{such that}
\begin{subequations}\label{p-epsgain}
\begin{multline}
\Big\| \Big(\sum_{k\in \bbZ}  \iint_{\cV^{n,l}_k}\Big| \int
K_k(\cdot-y,s,t) f_k(y) dy \Big|^2ds \,dt\Big)^{1/2}\Big\|_p \\ \le
C_p 2^{-(|n|+|l|)\eps(p)} \Big\|\Big(\sum_k|f_k|^2\Big)^{1/2}
\Big\|_p,
\end{multline}
\begin{multline}
\Big\| \Big(\sum_{k\in \bbZ}  \iint_{\cW^{n,\ell}_k}\Big| \int
K_k(\cdot-y,s,t) f_k(y) dy \Big|^2ds\, dt\Big)^{1/2}\Big\|_p\\ \le
C_p 2^{-(|n|+|l|)\eps(p)} \Big\|\Big(\sum_k|f_k|^2\Big)^{1/2}
\Big\|_p,
\end{multline}
\end{subequations}
and also the dual versions (with $\cH=L^2(\bbR\times\bbR)$)
\begin{subequations}\label{p-epsgain-dual}
\begin{multline}
\Big\| \Big(\sum_{k\in \bbZ} \Big| \iint_{\cV^{n,l}_k} \int
K_k(\cdot-y,s,t) F_k(y,s,t) dy \,ds\, dt\Big|^2\Big)^{1/2}\Big\|_p
\\ \le C_p 2^{-(|n|+|l|)\eps(p)}
\Big\|\Big(\sum_k|F_k|_{\cH}^2\Big)^{1/2} \Big\|_p,
\end{multline}
\begin{multline}
\Big\| \Big(\sum_{k\in \bbZ} \Big| \iint_{\cW^{n,\ell}_k}\int
K_k(\cdot-y,s,t) F_k(y,s,t) dy \,ds \,dt\Big|^2\Big)^{1/2}\Big\|_p
\\ \le C_p 2^{-(|n|+|\ell|)\eps(p)}
\Big\|\Big(\sum_k|F_k|_{\cH}^2\Big)^{1/2} \Big\|_p.
\end{multline}
\end{subequations}
For $p=2$ such estimates follow from \S\ref{L2sect}. For $p=1$ one
proves slightly weaker $L^1\to L^{1,\infty}$ inequalities, with
constants $O(1+|n|+|l|)$ and $O(1+|n|+|\ell|)$, respectively. These
follow if one checks the {H\"ormander condition on the
kernel~$K_{k}$,} \cf. \cite{hoerm} and \cite{stein-book}, namely
\begin{subequations}
\begin{multline}\label{VlHoer}
\int_{|x|\ge 2h} \Big(\sum_{k\in \bbZ} \iint_{\cV_k^{n,l}} \big|K_k(x+h,s,t)-K_k(x,s,t)|^2 ds\, dt\Big)^{1/2} dx \\
\lc  1+|n|+|l|
\end{multline}
and
\begin{multline}\label{WellHoer}
\int_{|x|\ge 2h} \Big(\sum_{k\in \bbZ} \iint_{\cW_k^{n,\ell}} \big|K_k(x+h,s,t)-K_k(x,s,t)|^2 ds\, dt\Big)^{1/2} dx \\
\lc 1+|n|+|\ell|  \,.
\end{multline}
\end{subequations}
In fact slightly better bounds than \eqref{VlHoer}, \eqref{WellHoer}
can be proved, but they are not good enough to sum in all the
parameters $(n,l)$, $(n,\ell)$, respectively. {Inequalities}
\eqref{VlHoer}, \eqref{WellHoer} can be established by
straightforward $L^1$ and $L^2$ estimates used earlier; we shall not
include the details. One can interpolate the  weak type $(1,1)$
inequalities implied by \eqref{VlHoer}, \eqref{WellHoer}
 and the {improved}  $L^2$ results  to {show} the $L^p$ inequalities
\eqref{p-epsgain} and \eqref{p-epsgain-dual}, and these yield a
proof of Theorem \ref{sobchar}.

%\section{Open Problems}\label{opensect}

%\subsection{\it }
%For $1/2<\alpha<3/2$, do we have the weak type $(1,1)$  inequality $\|S_\alpha f\|_{L^{1,\infty}} \lc \|D^\alpha f\|_{L^1}$?  Theorem \ref{dualH1thm} shows that the corresponding inequality with $\|\cD^\alpha f\|_{H^1}$ holds and, in \S\ref{weaktype11failure} we showed for $\alpha=1$ that we cannot replace $\|\cD^1 f\|_{L^1}$ with $\|f'\|_{L^1}$.

%\subsection{\it}  Is there a meaningful dual version to the estimate in Theorem \ref{dualH1thm}?
%   I.e., is there a function space $\bbV\subset L^\infty$ so that
%$$\|S_\alpha(\cD^{-\alpha} f)\|_{BMO} \lc \|f\|_{\bbV}$$ holds
%(or an analogous version with vector-valued $BMO$ holds) and so that the spaces $L^p$  for $2<p<\infty$ can be obtained as real interpolation spaces for the couple $(L^2, \bbV)$?

\section{Pointwise differentiability}\label{sec8}

Let $f\in L^{2}(\mathbb{R})$. {A} classical result of Stein and
Zygmund \cite{stein-zyg}, \cite[ch. VIII]{stein-book} says that $f$
is differentiable at almost every point~$x\in\mathbb{R}$ for which
there exists $\delta=\delta(x)>0$ such that
\begin{equation}\label{eq58}
\sup_{|t|<\delta} \left| \frac{f(x+2t)-f(x)}{2t}
-\frac{f(x+t)-f(x)}{t}\right| <\infty
\end{equation}
and
\begin{equation}\label{eq0}
\int_{|t|<\delta}\Big|\frac{f(x+2t)-f(x)}{2t}
-\frac{f(x+t)-f(x)}{t}\Big|^{2} \frac{dt}{|t|}<\infty.
\end{equation}
Conversely, for almost every point $x\in \mathbb{R}$ where $f$ is
differentiable there exists $\delta=\delta(x)>0$ such that
\eqref{eq0} holds. {Notice that \eqref{eq58} is the Zygmund
condition at $x$ in disguise.}
%See Chapter~VIII of~\cite{stein-book}.

The purpose of this section is to discuss analogous results when the
integral in~\eqref{eq0} is replaced by local versions of $S_\alpha
f$ for $\alpha=1${, the square function of the previous sections.}
We drop the subscript and write $S f\equiv S_1 f$.

\subsection{\it Preliminary considerations}
\subsubsection{\it Marcinkiewicz integrals}
The following classical  result on  Marcinkiewicz integrals is a
crucial tool in proving results on pointwise differentiation.

Let $F$ be a closed set of positive measure, and fix $\la>0$. Let
\[
I^{(\lambda)}(x):=\int^{x+1}_{x-1}
\frac{\operatorname{dist}^{\lambda}(y,F)}{|x-y|^{1+\lambda}}\,dy.
\]
Then one proves \cite[p.15]{stein-book} that \Be
\label{marcineq}I^{(\lambda)}(x)<\infty \text{ for almost every }
x\in F. \Ee
\subsubsection{Pointwise comparison with a   related square function}\label{ptfailure}
Given $f\in L^{2} (\mathbb{R})$ and $m\in\mathbb{R}$, consider the
 square functions $Qf$ defined for $x\in \bbR$ by
%\begin{align*}
%\langle f\rangle^{2}_{m}(x)
%\[\cQ_m f(x)=\Big( \int_{\mathbb{R}}\! \Big| \frac{f(x+mt)-f(x)}{mt} -\frac{f(x+t)-f(x)}{t}\Big|^2 \frac{dt}{|t|}\Big)^{1/2},\] and
\[
Qf(x)=\Big(\int_{1<|m|\le 2}\int_{\mathbb{R}} \Big|
\frac{f(x\!+\!mt)\!-\!f(x\!+\!t)}{(m-1)t} -
\frac{f(x\!+\!t)\!-\!f(x)}{t}\Big|^2
 \frac{dt}{|t|}\,dm\Big)^{1/2}.
\]
%&\;\; &x\!\in\! \mathbb{R}.\!\end{align*}
We shall use the identity
\begin{multline}\label{eq2}
\frac{f(x+mt)-f(x)}{mt} -\frac{f(x+t)-f(x)}{t}\\
= \frac{m-1}{m} \Big( \frac{f(x+mt)-f(x+t)}{(m-1)t}
-\frac{f(x+t)-f(x)}{t}\Big)
\end{multline}
to show that $Qf$ and $Sf$  are equivalent.

\begin{lemma}\label{lem8.1}
There exists a constant $C>0$ such that $$ C^{-1}Qf(x)\le Sf(x)\le C
Qf(x), \quad x\in \bbR, \quad f\in L^2(\bbR). $$
%, for any $f\in L^{2}(\mathbb{R})$.
\end{lemma}

\begin{proof}
Fix $u\in \mathbb{R}$ and $N\ge 1$. Since
\begin{multline*}
\frac{f(x+u^{N}t)-f(x)}{u^{N}t} -\frac{f(x+t)-f(x)}{t}\\
=
\sum^{N}_{j=1}\frac{f(x+u^{j}t)-f(x)}{u^{j}t}-\frac{f(x+u^{j-1}t)-f(x)}{u^{j-1}t},
\end{multline*}
 Schwarz's inequality gives
\begin{multline*}
\left| \frac{f(x+u^{N}t)-f(x)}{u^{N}t} -\frac{f(x+t)-f(x)}{t}\right|^{2}\\
\le N\sum^{N}_{j=1}\Big| \frac{f(x+u^{j}t)-f(x)}{u^{j}t}
-\frac{f(x+u^{j-1}t)-f(x)}{u^{j-1}t}\Big|^{2}.
\end{multline*}
Hence, with $\cG_{1,m}$ as in \eqref{cGdef},
\begin{equation*}
\begin{split} \cG_{1, u^N}f(x)^2
%\langle f\rangle^{2}_{u^{N}}(x)
&\le N\sum^{N}_{j=1}\int_{\mathbb{R}} \Big|\frac{f(x+u^{j}t)-f(x)}{u^{j}t} -\frac{f(x+u^{j-1}t)-f(x)}{u^{j-1}t}\Big|^{2}\frac{dt}{|t|}\\
%*[7pt]
&=N^{2}\cG_{1,u}f(x)^2.
%\langle f\rangle^{2}_{u}(x).
\end{split}
\end{equation*}
Fix $0<\varepsilon<1$. We perform the change of variable $m=u^{N}$ and then estimate
\begin{align*}
&\int^{2^{N}}_{(1+\varepsilon)^{N}} \cG_{1,m}f(x)^2\frac{dm}{(m-1)^2} 
= \int_{1+\eps}^2 \cG_{1,u^N} f(x)^2 \frac{Nu^{N-1}}{(u^N-1)^2} du
\\
&\le N \int_{1+\eps}^2 \cG_{1,u^N} f(x)^2 \frac{1}{u^{N-1}} \frac{du}{(u-1)^2} 
\le \frac{N^3}{(1+\eps)^{N-1}}  \int_{1+\eps}^2 \cG_{1,u} f(x)^2 \frac{du}{(u-1)^2} .
\end{align*}
%\frac{dm}{(1-m)^{2}}\le \frac{N^{3}}{(1+\varepsilon)^{N-1}}
%\int^{2}_{1+\varepsilon} \cG_{1,m}f(x)^2
%\frac{dm}{(1-m)^{2}}.$$

A similar argument gives
$$
\int_{-2^{N}}^{-(1+\varepsilon)^{N}} \cG_{1,m}f(x)^2
%\langle f\rangle^{2}_{m}(x)
\frac{dm}{(m-1)^{2}}\le \frac{N^{3}}{(1+\varepsilon)^{N-1}}
\int_{-2}^{-1-\varepsilon} \cG_{1,u}f(x)^2
%\langle f\rangle^{2}_{m}(x)
\frac{du}{(u-1)^{2}}.
$$
Thus
$$
\int_{|m|\ge 2}\cG_{1,m}f(x)^2
%\langle f\rangle^{2}_{m}(x)
\frac{dm}{(m-1)^{2}} \lesssim \int_{1<|m|\le 2} \cG_{1,m}f(x)^2
%\langle f\rangle^{2}_{m}(x)
\frac{dm}{(m-1)^{2}}.
$$
By the identity \eqref{eq1} we get
$$
Sf(x)^2\approx
%\simeq
\int_{1<|m|\le 2}\cG_{1,m}f(x)^2
%\langle f\rangle^{2}_{m}(x)
\frac{dm}{(m-1)^{2}}
$$
and the asserted equivalence  follows immediately from the identity
\eqref{eq2}.
\end{proof}

\subsubsection{An inequality for functions in the Zygmund class}
For the proof of Theorems \ref{ptwintro} and \ref{theo8.2}  we need
the following.
\begin{lemma}\label{zyglemma}
Let $f\in\La_*$. Then there is a constant $C$ such that
\begin{multline}
\sup_{x\in \bbR} \sup_{|t|\le 1}
\Big|\frac{f(x+mt)-f(x)}{mt} - \frac{f(x+t)-f(x)}{t}\Big| \\
\le C\|f\|_{\La_*}|m-1| (1+\log(|m-1|^{-1}))\,,
\end{multline}
for $1 < |m|\le 2$.
\end{lemma}
\begin{proof}
We shall use that divided differences of functions in $\La^*$
satisfy a mild regularity property, namely
\begin{equation}\label{DLlN}
\Big|\frac {f(x+t)-f(x)}{t}-\frac{f(x+s)-f(x)}{s}\Big| \le C
\|f\|_{\La_*}
\end{equation}
{for  $x,t,s\in \bbR$ with $|s|/2\le |t|\le |s|$,} see \cite[Lemma
2]{DLlN}. This implies in particular the easier version of
\eqref{eq6} where the sup is just taken over $m\in [-2,-1]$.

Now if $1<m\le 2$ we apply the crucial identity \eqref{eq2} to gain
the factor $m-1$; we then see that it suffices  to show, for any
$x,t\in \bbR$ and $1<m\le 2$, \Be\label{m-cond} \Big|
\frac{f(x+mt)-f(x+t)}{(m-1)t} -\frac{f(x+t)-f(x)}{t}\Big|\lc
\|f\|_{\La_*}(1+\log \frac{1}{m-1}). \Ee Let $N$ be the positive
integer satisfying $1<2^N(m-1)\le 2$. Since
\[
\Big| \frac{f(x+t+2^{k-1}(m-1)t))-f(x+t)}{2^{k-1}(m-1)t}
-\frac{f(x+t+2^k(m-1)t)-f(x+t)}{2^k(m-1)t}\Big|
\]
is bounded by $C \|f\|_{\La_*}$, uniformly in $x,t,k,m$, we obtain,
summing in $k=1,\dots, N$,
\[
\Big| \frac{f(x+mt)-f(x+t)}{(m-1)t}
-\frac{f(x+t+2^N(m-1)t)-f(x+t)}{{2^{N}}(m-1)t}\Big|\lc N
\|f\|_{\La_*}.
\]
This gives \eqref{m-cond}.
\end{proof}

\subsection{\it Differentiability versus  finiteness of a square-function: an example}
We shall consider for any $\delta>0$  the local version
$S_{\text{loc},\delta}$ of $S$, defined by
%$$S^{2}_{\text{loc}}f(x)=\!\iint_{|t|+|s|<\delta} \left( \frac{f(x+t)-f(x)}{t} -\frac{f(x+s)-f(x)}{s}\right)^{2}\frac{ds\,dt}{|s-t|^{2}}=+\infty$$
\begin{multline}\label{Slocdef}
S_{\text{loc},\delta}f(x)\\
= \Big(\iint_{|t|+|s|<\delta} \Big| \frac{f(x+t)-f(x)}{t}
-\frac{f(x+s)-f(x)}{s}\Big|^{2}\frac{ds\,dt}{|s-t|^{2}}\Big)^{1/2}.
%=+\infty
\end{multline}

%\footnote{CHECK}
We show that the  finiteness of $S_{\text{loc},\delta}f(x)$ is
generally not a necessary condition for differentiability.
Specifically we  present an example of a function~$f$ differentiable
at almost every point of a set~$E$ of positive measure such that for
any $\delta>0$,
%the local version~$S_{\text{loc},\delta}$ of $$ satisfies
\[
S_{\text{loc},\delta}f(x)= \infty, \text{ for a.e. } x\in E\,.\]
%\!\iint_{|t|+|s|<\delta} \left( \frac{f(x+t)-f(x)}{t} -\frac{f(x+s)-f(x)}{s}\right)^{2}\frac{ds\,dt}{|s-t|^{2}}=+\infty
%a.e.\ $x\in E$, for any $\delta>0$.
Hence an analogue of the result of Stein and Zygmund in this context
does not hold without additional assumptions on the function~$f$
(such as for example the Zygmund class condition in Theorem
\ref{ptwintro}).

Let $E\subset \mathbb{R}$ be a closed set of  positive Lebesgue
measure without interior points.
 Write $\mathbb{R}\backslash E =\cup I_{j}$, where $I_{j}=(c_j-b_j, c_j+b_j)$
 are pairwise disjoint open intervals.
We denote by $I_j^{\text{half}} =(c_j-b_j/2, c_j+b_j/2)$ the inner
half of $I_j$. Let $f\colon \mathbb{R}\to \mathbb{R}$ satisfying
\begin{equation}\label{eq3}
|f(x)|\le \operatorname{dist}(x,E),\quad x\in\mathbb{R}
\end{equation}
and,  for each $j$,
\begin{equation}\label{eq3bis}
\int_{0}^\delta\left|\frac{f(y+u)-f(y)}{u}\right|^{2}\,du=\infty
\text{ for all } y\in I_j^{\text{half}}.
%, \,\,j=1,2,\dots
\end{equation}
%for any $y\in I_{j}/2$, $j=1,2,\dotsc$. Here $I/2$ denotes the interval centered at the center of~$I$ whose length is half the length of $I$.

The change of variable $s=mt$ and identity~\eqref{eq3bis} gives
%$S^{2}_{\text{loc},\delta}f(x)\approx $
\begin{multline*}
S^{2}_{\text{loc},\delta}f(x)\\
\approx \int_{|t|<\delta} \int_{|m|>1}\left(
\frac{f(x+mt)-f(x+t)}{(m-1)t}
-\frac{f(x+t)-f(x)}{t}\right)^{2}\,dm\frac{dt}{|t|}.
\end{multline*}
We apply  now Stepanov's Theorem
%(Theorem 3 in chapter VIII of
(\cite[VIII, Thm.3]{stein-book} or \cite{maly}). It says that $f$ is
differentiable at almost every point in $E$ if and only if
$f(x_0+y)-f(x_0)=O(|y|)$ as $|y|\to 0$ for almost every $x_0\in E$.
Hence condition~\eqref{eq3} implies that $f$ is differentiable at
almost every point of~$E$. Moreover
 \eqref{eq3}  and the Marcinkiewicz inequality \eqref{marcineq} for $\la=2$
imply
% a  result  on Marcinkiewicz integrals in \cite[p.16]{stein-book} says
%$$ \int_{|t|\le 1}  \frac {[\dist(x+t,E)]^2}{t^3} dt <\infty \text{ for almost every $x\in E$}$$ and thus
\[\int_{|t|<\delta}\left(\frac{{f(x+t)}}{t}\right)^{2}\frac{dt}{|t|}<\infty,\quad \text{for a.e.\ }x\in E.
\]
On the other hand, the change of variable~$(m-1)t=u$ gives
\begin{multline*}
\int_{|t|<\delta}\int_{|m|>1}\left( \frac{f(x+mt)-f(x+t)}{(m-1)t} \right)^{2}\,dm\frac{dt}{|t|}\\
\ge \int_{|t|<\delta} \int_{0}^{1}\left(
\frac{f(x+t+u)-f(x+t)}{u}\right)^{2}\,du\frac{dt}{|t|^{2}}.
\end{multline*}
Now, for fixed $\delta>0$, for almost every $x\in E$ the interval
$(x-\delta, x+\delta)$ contains an interval~$I_{j}$. Here we use the
assumption that $E$ is a closed set with no  interior points.  Hence
there exists a set of points $t\in (-\delta,\delta)$ of positive
measure such that $x+t\in I_{j}^\text{half}$ and
condition~\eqref{eq3bis} shows that the last  integral diverges.
Hence $S_{\text{loc},\delta}f(x)=\infty$ for almost every  $x\in E$.

\subsection{\it The main result on pointwise differentiability}
We shall now consider functions that are locally in the Zygmund
class, i.e. satisfy condition \eqref{eq4} below. {This condition
clearly holds} when $f$ is differentiable at $x$, but {it is}
substantially weaker.
% than differentiability.

\begin{theorem}\label{theo8.2}
Let $f\in L^{2}_{\operatorname{loc}}(\mathbb{R})$.
%\begin{enumerate}
%\item[(a)]

a) The function $f$ is differentiable at almost every
point~$x\in\mathbb{R}$ where the following two conditions hold
\begin{equation}\label{eq4}
\limsup_{|h|\to 0}\left|\frac{f(x+2h)-f(x)}{2h}-
\frac{f(x+h)-f(x)}{h}\right|<\infty
\end{equation}
and there exists $\delta=\delta(x)>0$ such that
\begin{equation}\label{eq5}
\iint_{|s|+|t|<\delta} \left| \frac{f(x+s)-f(x)}{s} -
\frac{f(x+t)-f(x)}{t}\right|^{2}\frac{ds\,dt}{|s-t|^{2}}<\infty.
\end{equation}

%\item[(b)]
b) For almost every point $x\in\mathbb{R}$ where $f$ is
differentiable and
\begin{equation}\label{eq6}
\limsup_{|t|\to 0} \sup_{1<|m|\le 2}\frac{ \left|
\frac{f(x+mt)-f(x)}{mt}-\frac{f(x+t)-f(x)}{t}\right|}{|m-1|(1+\big|\log
\tfrac{1}{|m-1|}\big|)}<\infty,
\end{equation}
there exists $\delta=\delta(x)>0$ such that
\begin{equation}\label{eq7}
\iint_{|s|+|t|<\delta} \left| \frac{f(x+s)-f(x)}{s} -
\frac{f(x+t)-f(x)}{t}\right|^{2}\frac{ds\,dt}{|s-t|^{2}}<\infty.
\end{equation}
%\end{enumerate}
\end{theorem}

\begin{proof}[Proof of Theorem \ref{ptwintro}]
One direction is immediate {by a) of Theorem~\ref{theo8.2}}. For the
other direction one  needs  to verify that condition \eqref{eq6}
holds for any $f\in \Lambda_*$.  But this  was proved in Lemma
\ref{zyglemma}. \end{proof}
%Functions in the Zygmund class may be nowhere differentiable. For instance the Weierstrass nowhere differentiable function
%$$f(x)=\sum^{\infty}_{n=1}b^{-n} \cos (b^{n}x),\quad x\in \mathbb{R},$$
%where $b>1$, is in the Zygmund class.

%\begin{corollary}Let $f\in \Lambda_{*}$. Then the set of points $x\in\mathbb{R}$ such that$$\iint_{|t|+|s|<1} \left( \frac{f(x+s)-f(x)}{s} -\frac{f(x+t)-f(x)}{t}\right)^{2}\frac{ds\,dt}{|s-t|^{2}}<  \infty$$coincides except possibly for a set of Lebesgue measure zero, with the set of points where $f$ is differentiable.\end{corollary}

\begin{proof}[Proof of Theorem \ref{theo8.2}]
(a) Given $\delta>0$, let $E=E^\delta$ be the set of points $x\in
\mathbb{R}$ for which
$$
\iint_{|t|+|s|<\delta} \left| \frac{f(x+t)-f(x)}{t}-
\frac{f(x+s)-f(x)}{s}
\right|^{2}\frac{ds\,dt}{|s-t|^{2}}<\delta^{-1}
$$
and
$$
\sup_{|h|<\delta}\left|\frac{f(x+2h)-f(x)}{2h}
-\frac{f(x+h)-f(x)}{h}\right|<\delta^{-1}.
%\quad |h|<\delta.
$$
We show that for any fixed $\delta>0$, the function~$f$ is
differentiable at almost every point of~$E^\delta$. We can assume
that $f$ vanishes outside an interval~$I$ of length~$\delta$. For
$x\in E^\delta\cap I$ we have $Sf(x)<\infty$ and 
Lemma~\ref{lem3.1} gives $G_1f(x)=2\cG_{1,2}f(x)<\infty$.
%$\langle f\rangle_{2}(x)<\infty$.
Now, the Stein--Zygmund result gives that $f$ is differentiable at
almost every point of $E^\delta\cap I$. The assertion a) follows if
we consider the union  $\cup_{j>1} E^{1/j}$.

\vspace*{7pt}

(b) Given $\delta>0$, let $E(\delta)$ be the set of points $x\in
\mathbb{R}$ for which
%the following two conditions hold
%\begin{alignat}{2}
\[
|f(x+t)-f(x)| \le \delta^{-1}|t|,\] holds for $0\le |t| \le\delta$
and
\[\Big| \frac{f(x\!+\!mt)\!-\!f(x)}{mt}- \frac{f(x\!+\!t)\!-\!f(x)}{t}\Big |\le \delta^{-1}|m\!-\!1|\big|{1+}\log \tfrac{1}{|m\!-\!1|}\big|\]
holds when  $1<|m|\le 2,$ and $0<|t|\le\delta$. It suffices to show
that condition~\eqref{eq7} holds for almost every point $x\in
E(\delta)$ for each given $\delta$ (then one takes the  union
$\cup_j E(1/j)$) Without loss of generality we can assume that
$E(\delta)$ is compact and that $f$ vanishes outside an interval~$I$
of length~$\delta$.

Given $\varepsilon >0$, we prove that the set of all $x\in
E(\delta)$ where \eqref{eq7} fails is of measure less than $\eps$.
We can find a compact set $F\subset E(\delta)$ {with}
$|E(\delta)\backslash F|<\varepsilon$ and a decomposition $f=g+b$
where $g$ is Lipschitz on $\bbR$ and $b$ vanishes on~$F$
(\cite[p.~248]{stein-book}). Moreover we can also assume that $g$
and $b$ vanish outside $I^*$,
%here and in what follows  we denote by $I^*$
 the double interval with the same center.
Applying the $L^2$ inequality for $Sg$ we get  $Sg(x)<\infty$ for
almost every $x\in \mathbb{R}$. Hence we need to show that
\[Sb(x)<\infty \text{  for almost every $x\in F \cap I$}.\]

Since $g$ is Lipschitz on $\bbR$ we get \Be \sup_{x\in
\bbR}\sup_{|t|\le \delta}\frac{ |g(x+t)-g(x)| }{|t|}\!\le\!
C_1,\label{eq8}\Ee and since the Lipschitz space is contained in the
Zygmund class we also have by Lemma \ref{zyglemma} \Be \sup_{x\in
\bbR}\sup_{\substack{ |t|\le \delta}} \Big| \frac{g(x+mt)-b(x)}{mt}-
\frac{g(x\!+\!t)\!-\!b(x)}{t}\Big| \le C_2|m-1| (1+\log
\tfrac{1}{|m-1|} )\label{eq9}\Ee for  $1\le|m|\le 2$.

Therefore
 the function~$b$ satisfies, for some positive constant $A$,
   \Be \sup_{|t|\le \delta}\frac{ |b(x+t)-b(x)| }{|t|}\!\le\!
A,\label{eq8}\Ee and \Be \sup_{\substack{ |t|\le \delta}} \Big|
\frac{b(x+mt)-b(x)}{mt}- \frac{b(x\!+\!t)\!-\!b(x)}{t}\Big| \le A
|m-1| (1+\log \tfrac{1}{|m-1|} )\label{eq9}\Ee for all $x\in F$ and
$1\le|m|\le 2$.

For the remainder of this proof implicit constants in inequalities
of the form $\lc$ may depend on $A$.
%Let us recall a classical result of Marcinkiewicz. Fix $\lambda >0$. Then$$I^{(\lambda)}(x)=\int^{x+1}_{x-1}
% \frac{\operatorname{dist}^{\lambda}(y,F)}{|x-y|^{1+\lambda}}\,dy<\infty
 %$$for almost
% every $x\in F$. See p.~15 of~\cite{stein-book}.

Consider a Whitney decomposition of the open set~$I^*\backslash
\overline{I}\cap F$, that is, $I^*\backslash \overline{I}\cap F=\cup
I_{j}$, where $\{I_{j}\}$ are pairwise disjoint intervals with
\[|I_j| \le \operatorname{dist}(I_{j},\overline{I}\cap F)\le 4 |I_{j}|.\]
Set $\overline I_j=[a_j,b_j]$ and let $x_{j}=(a_j+b_j)/2$ denote the
center of~$I_{j}$. We let
%Note that the double interval is
%We shall also write i.e.
 $I_j^*= (x_j-|I_j|, x_j+|I_j|)$ denote  the open double interval.
By \eqref{marcineq}
\begin{equation}\label{eq10}
\sum_{j}\frac{|I_{j}|^{1+\lambda}}{|x-x_{j}|^{1+\lambda}} \lesssim
I^{\lambda} (x)<\infty
\end{equation}
for almost every $x\in I\cap F$. The plan of the proof is to show
that there exists $\lambda>0$ such that
$$
|Sb(x)|^2\lesssim
\sum_{j}\frac{|I_{j}|^{1+\lambda}}{|x-x_{j}|^{1+\lambda}},
$$
for almost every $x\in F\cap I$.

By part~(a) of Lemma~\ref{lem8.1} we have $Sb(x)\approx Qb(x)$.
Write \[L=\{m\in\mathbb{R}: 1<|m|\le 2\}.\] Then $|Qb(x)|^2\lesssim
A(x)+B(x),$ where
\begin{align*}
A(x)&=\int_{\mathbb{R}} \left|\frac{b(x+t)-b(x)}{t}\right|^{2}
\frac{dt}{|t|},\\*[7pt] B(x)&=\int_{L}\int_{\mathbb{R}}
\left|\frac{b(x+mt)-b(x+t)}{(m-1)t}\right|^{2}\frac{dt}{|t|}\,dm.
\end{align*}
We need to show that $A(x)<\infty$, $B(x)<\infty$ for almost every
$x\in F\cap I$. In what follows we will always assume $x\in F\cap
I$.

Since $b_{|F}\equiv 0$, condition~\eqref{eq8} gives that
$\sup\limits_{x\in I_{j}} |b(x)|\lesssim |I_{j}|$ and hence
\begin{equation}\label{eq11}
\int_{I_{j}}|b|^{2}\lesssim |I_{j}|^{3}.
\end{equation}
{Therefore}
%for $x\in F\cap I$
$$
A(x)=\int_{\mathbb{R}}\frac{b(y)^{2}} {|y-x|^{3}}\,dy\lesssim
\sum_{j}\frac{|I_{j}|^{3}}{|x_{j}-x|^{3}}
$$
which by~\eqref{eq10} is finite for almost every $x\in F\cap I$.

Now $B(x)=C(x)+D(x)$ where
\begin{align*}
C(x)&=\int_{L}\int_{F-x}\Big|\frac{b(x+mt)}{(m-1)t}\Big|^{2}\frac{dt}{|t|}\,dm,
\\*[7pt]
D(x)&=\int_{L}\int_{\mathbb{R}\backslash
(F-x)}\Big|\frac{b(x+mt)-b(x+t)}{(m-1)t}\Big|^{2}\frac{dt}{|t|}\,dm.
\end{align*}
For each $m\in L$ and $j=1,2,\dotsc$, let
$I_{j}(m)=\frac{I_{j}-x}{m}$. We have
%$ be the interval~$I_{j}(m)=\frac{I_{j}-x}{m}$. We have
$$
C(x)=\int_{L}\sum_{j}\int_{t\in (F-x)\cap
I_{j}(m)}\Big|\frac{b(x+mt)}{(m-1)t}\Big|^{2}\frac{dt}{|t|}\,dm.
$$
Since $b_{|F}\equiv 0$, condition~\eqref{eq8} gives that
$\sup\{|b(x+mt)|: t\in I_{j}(m)\}\lesssim |I_{j}|$. Hence
$$
C(x)\lesssim\int_{L}\sum_{j} |I_{j}|^{2}\Big(\int_{t\in (F-x)\cap
I_{j}(m)}\frac{dt}{|t|^{3}}\Big)\frac{dm}{(m-1)^{2}}.
$$
Now for each $t\in F-x$ we have $t\in I_{j}(m)$ if and only if $m\in
I_{j}(t)$. Write $I_{j}=(a_{j},b_{j})$. Then
$I_{j}(t)=((a_{j}-x)/t,(b_{j}-x)/t)$ and
$$
\int_{I_{j}(t)}\frac{dm}{(m-1)^{2}}\le \left|\frac{t}{a_{j}-x-t}
-\frac{t}{b_{j}-x-t}\right|
=|t|\frac{|I_{j}|}{|a_{j}-x-t||b_{j}-x-t|}.
$$
Since $x+t\in F$ we have that both $|a_{j}-(x+t)|$ and
$|b_{j}-(x+t)|$ are comparable to $|x_{j}-(x+t)|$. Then
$$
\int_{I_{j}(t)} \frac{dm}{(m-1)^{2}}\lesssim \frac{|t||I_{j}|}
{|x_{j}-(x+t)|^{2}}.
$$
Since $x+t\in F$ we have $|x+t-x_{j}|\gtrsim |I_{j}|$. Moreover
{from} $m\in I_{j}(t)$, $|m|\le 2$ and $x\in F$ we have $|t|\ge
|a_{j}-x|/2\approx |x-x_{j}|$. Fubini's Theorem gives
\begin{multline*}
C(x)\lesssim\sum_{j}|I_{j}|^{3}\int_{\begin{subarray}{c} |t|\gtrsim
|x-x_{j}|\\ |x+t-x_{j}|\gtrsim
|I_{j}|\end{subarray}}\frac{dt}{|x_{j}-(x+t)|^{2}|t|^{2}}\\*[7pt]
\lesssim
\sum_{j}\frac{|I_{j}|^{3}}{|x-x_{j}|^{2}}\int_{|x+t-x_{j}|\gtrsim
|I_{j}|} \frac{dt}
{|x_{j}-(x+t)|^{2}}\lesssim\sum_{j}\frac{|I_{j}|^{3}}{|x-x_{j}|^{3}},
\end{multline*}
which by \eqref{eq10} is finite a.e.\ $x\in F$.

Next we  estimate $D(x)$. By Fubini's Theorem,
$$
D(x)=\sum_{j}\int^{b_{j}-x}_{a_{j}-x} \int_{L}\Big|
\frac{b(x+mt)-b(x+t)}{(m-1)t} \Big|^{2}\,dm \frac{dt}{|t|}=
\sum_{j}E_{j}(x)+F_{j}(x),
$$
where
\begin{align*}
E_{j}(x)&=\int^{b_{j}-x}_{a_{j}-x}\int_{\{m\in L: x+mt\in F\}}\Big|
\frac{b(x+t)}{(m-1)t}\Big|^{2}\,dm \frac{dt}{|t|},\\*[7pt]
F_{j}(x)&=\int^{b_{j}-x}_{a_{j}-x}\int_{\{m\in L: x+mt\notin
F\}}\Big| \frac{b(x+mt)-b(x+t)}{(m-1)t}\Big|^{2}\,dm \frac{dt}{|t|}.
\end{align*}
Note that since $\{I_{j}\}$ are Whitney intervals, $x+t\in I_{j}$
and $x+mt\in F$ implies that $|m-1|\gtrsim |I_{j}|/|x-x_{j}|$. Hence
$$
E_{j}(x)\lesssim
\int^{b_{j}-x}_{a_{j}-x}\frac{|b(x+t)|^2}{|t|^{3}}\frac{|x-x_{j}|}{|I_{j}|}\,dt.
$$
Since $b_{|F}\equiv 0$ condition~\eqref{eq8} gives that
$|b(x+t)|\lesssim |I_{j}|$ for $t\in (a_{j}-x,b_{j}-x)$. We have
$|t|\approx |x-x_{j}|$ for any $t\in (a_{j}-x,b_{j}-x)$ {because}
the $I_{j}$ are Whitney intervals.

It follows that
$$
E_{j}(x)\lesssim \frac{|I_{j}|^{2}}{|x-x_{j}|^{2}}
$$
and
$$
\sum_{j}E_{j}(x)\lesssim \sum_{j}\frac{|I_{j}|^{2}}{(x-x_{j})^{2}},
$$
which by \eqref{eq10} is finite a.e.\ $x\in F$.

Let us now estimate $F_{j}(x)$. Split $F_{j}(x)=G_{j}(x)+H_{j}(x)$
where
\begin{align*}
G_{j}(x)&=\int^{b_{j}-x}_{a_{j}-x}\int_{\{m\in L: x+mt\in
I_{j}^*\}}\Big|\frac{b(x+mt)-b(x+t)}{(m-1)t}\Big|^{2}\,dm
\frac{dt}{|t|},\\*[7pt] H_{j}(x)&=\int^{b_{j}-x}_{a_{j}-x}\sum_{k\ne
j}\int_{\{m\in L: x+mt\in I_{k}\backslash
I_{j}^*\}}\Big|\frac{b(x+mt)-b(x+t)}{(m-1)t}\Big|^{2}\,dm
\frac{dt}{|t|}.
\end{align*}
Recall that $I_j^*$ denotes the interval of double length centered
at $x_j$. Now, the assumption~\eqref{eq6} and identity~\eqref{eq2}
give that
$$
\Big| \frac{b(x+mt)-b(x+t)} {(m-1)t} -\frac{b(x+t)-b(x)}{t}\Big|
\lesssim \Big|\ln \frac{1}{|m-1|}\Big|.
$$
Hence $G_{j}(x)$ is bounded by a fixed multiple of
$$
\int^{b_{j}-x}_{a_{j}-x}\int_{\{m\in L: x+mt\in I_{j}^*\}}
\ln^{2}\frac{1}{|m-1|} \,dm \frac{dt}{|t|}
+\int_{a_{j}-x}^{b_{j}-x}\Big|
\frac{b(x+t)-b(x)}{t}\Big|^{2}\frac{dt}{|t|}.
$$
Since $\{I_{j}\}$ are Whitney intervals, the fact that $x+t\in
I_{j}$ and $x+mt\in  I_{j}^*$ {implies} that $|m-1|\lesssim
\frac{|I_{j}|}{|x-x_{j}|}$. {So}
$$
\int_{\{m\in L: x+mt\in I_{j}^*\}} \ln^{2}\frac{1}{|m-1|} \,dm
\lesssim \frac{|I_{j}|}{|x-x_{j}|}\ln^{2}\frac{|I_{j}|}{|x-x_{j}|}
$$
and we deduce that
$$
\int^{b_{j}-x}_{a_{j}-x}\int_{\{m\in L: x+mt\in I_{j}^*\}}
\ln^{2}\frac{1}{|m-1|} \,dm\frac{dt}{|t|} \lesssim
\frac{|I_{j}|^{2}}{|x-x_{j}|^{2}}\ln^{2}\frac{|I_{j}|}{|x-x_{j}|}.
$$
{Because} $\{I_{j}\}$ are Whitney intervals $|I_{j}|\lesssim
|x-x_{j}|$ and $\ln^{2}(|I_{j}|/|x-x_{j}|)\lesssim
|x-x_{j}|^{\alpha}/|I_{j}|^{\alpha}$ for any~$\alpha>0$, we deduce
$$
\int^{b_{j}-x}_{a_{j}-x}\int_{\{m\in L: x+mt\in I_{j}^*\}}
\ln^{2}\frac{1}{|m-1|} \,dm\frac{dt}{|t|} \lesssim
\frac{|I_{j}|^{2-\alpha}}{|x-x_{j}|^{2-\alpha}}.
$$
As before
$$
\int^{b_{j}-x}_{a_{j}-x}\Big|\frac{b(x+t)-b(x)}{t}\Big|^{2}\frac{dt}{|t|}
\lesssim \frac{|I_{j}|^{3}}{|x-x_{j}|^{3}}.
$$
{We obtain}
$$
\sum_{j}G_{j}(x)\lesssim
\sum_{j}\frac{|I_{j}|^{2-\alpha}}{|x-x_{j}|^{2-\alpha}},
$$
where $0<\alpha<1$, which by \eqref{eq10}  is finite a.e.\ $x\in F$.

It remains to estimate  $H_{j}(x) $ for $x\in F\cap I$. Observe that
$|t|\sim |x-x_{j}|$ if $t\in (a_{j}-x,b_{j}-x)$. Since $1<|m|\le 2$
if the set
\[J_{k}=\{m\in L: x+mt\in I_{k}\backslash I_{j}^*\}\] is nonempty we {get} that $|x-x_{k}|\simeq |x-x_{j}|$. Now we have
\begin{align*}
\int_{J_{k}}\frac{dm}{(m-1)^{2}}&\lesssim
%\Big| \frac{t}{a_{k}-x-t}-\frac{t}{b_{k}-x-t}\Big|
\Big| \Big(\frac{a_k-x}{t}-1\Big)^{-1}
-\Big(\frac{b_k-x}{t}-1\Big)^{-1}\Big|
%\frac{t}{a_{k}-x-t}-\frac{t}{b_{k}-x-t}\Big|
\\ &=
 \frac{|t||b_k-a_k|}{|a_k-x-t||b_k-x-t|}
\approx \frac{|t| |I_{k}|}{(x_{k}-x_{j})^{2}}.
\end{align*}
Since $b_{|F}\equiv 0$, condition~\eqref{eq8} gives
$|b(x+mt)|\lesssim |I_{k}|$ and $|b(x+t)|\lesssim |I_{j}|$ for any
$t\in (a_{j}-x, b_{j}-x)$ and $m\in J_{k}$. Now, Fubini's Theorem
gives
$$
H_{j}(x)\lesssim \sum_{k\ne j: |x-x_{k}|\approx |x-x_{j}|}
(|I_{k}|^{2}+|I_{j}|^{2})\frac{|I_{k}|}{(x_{k}-x_{j})^{2}}
\int^{b_{j}-x}_{a_{j}-x}\frac{dt}{|t|^{2}}.
$$

Denote by $A(j)$ the set of indices~$k\ne j$ such that
$|x-x_{k}|\approx |x-x_{j}|$. Since $\{I_{j}\}$ are Whitney
{intervals}
$$
H_{j}(x)\lesssim \sum_{k\in
A(j)}\frac{(|I_{k}|^{2}+|I_{j}|^{2})|I_{k}| |I_{j}|}
{(x_{k}-x_{j})^{2}(x-x_{j})^{2}}
$$
{and since}
$$
\sum_{k\in A(j)}\frac{|I_{k}|} {(x_{k}-x_{j})^{2}}\lesssim
\frac{1}{|I_{j}|},
$$ we get
$$
\sum_{k\in A(j)}
\frac{|I_{j}|^{3}|I_{k}|}{(x_{k}-x_{j})^{2}(x-x_{j})^{2}}\lesssim
\frac{|I_{j}|^{2}}{(x-x_{j})^{2}}.
$$
{We have}
$$
\sum_{j}\sum_{k\in
A(j)}\frac{|I_{k}|^{3}|I_{j}|}{(x_{k}-x_{j})^{2}(x-x_{j})^{2}}\cong
\sum_{k}\sum_{j\in A(k)} \frac{|I_{k}|^{3} |I_{j}|}
{(x_{k}-x_{j})^{2}(x-x_{k})^{2}}
$$
%and the right hand side is estimated by
%$\sum_{k}\frac{|I_{k}|^{2}}{(x-x_{k})^{2}}$ by the estimate just given  before.
{and}
$$\sum_{j\in A(k)} \frac{|I_{j}|} {(x_{k}-x_{j})^{2}}\lesssim \frac{1}{|I_{k}|}. $$
{Therefore}
$$\sum_{j}\sum_{k\in A(j)} \frac{|I_{k}|^{3} |I_{j}|}{(x_{k}-x_{j})^{2}(x-x_{j})^{2}}\lesssim \sum_{k}\frac{|I_{k}|^{2}}{(x-x_{k})^{2}}$$
{and so}
$$
\sum_{j}H_{j}(x) \lesssim \sum_{j}\frac{|I_{j}|^{2}}{(x-x_{j})^{2}}
$$
which by \eqref{eq10} is finite a.e.\ $x\in F$. This finishes the
proof.
\end{proof}

% \subsection*{\it Acknowledgements.} 

\smallskip

\noi{\it Acknowledgements.}

{J.~Cuf\'{\i}, A.~Nicolau and J.~Verdera were partially supported by the
grants 2014SGR75 and 2014SGR289 of Generalitat de Catalunya,
MTM2014-51824, MTM2013-44699, MTM2016-75390 and MTM2017--85666 of Ministerio de Educaci\'on,
Cultura y Deporte.} 

{A.~Seeger was partially supported by  NSF grant DMS 1500162, and  as a Simons Visiting Researcher at CRM. He  would like to 
 thank  the  organizers of the 2016 program in Constructive Approximation and Harmonic Analysis
 for the invitation and for providing a pleasant and fruitful research atmosphere.}

\end{document}